\documentclass[12pt]{article}
\usepackage{amsmath,amsfonts,amsthm,amssymb}
\usepackage{subcaption}
\usepackage{hyperref}
\hypersetup{
    colorlinks=true,
    allcolors= {violet},
    menucolor={black},
}
\usepackage{color, fullpage}
\usepackage{tikz,varwidth}

\usetikzlibrary{calc}
\usetikzlibrary{shapes,arrows}
\usepackage{comment}
\usepackage[font={small,it},labelfont=sc]{caption}
\usepackage{mathtools}

\usetikzlibrary{positioning,calc}

\newcommand{\udensdash}[1]{%
    \tikz[baseline=(todotted.base)]{
        \node[inner sep=1pt,outer sep=0pt] (todotted) {#1};
        \draw[thick, densely dashed] (todotted.south west) -- (todotted.south east);
    }%
}%

\newtheorem{theorem}{Theorem}
\newtheorem{proposition}[theorem]{Proposition}
\newtheorem{lemma}[theorem]{Lemma}
\newtheorem{corollary}[theorem]{Corollary}
\theoremstyle{definition}
\newtheorem{definition}[theorem]{Definition}%added by Sheila
\newtheorem{example}[theorem]{Example}
\theoremstyle{remark}
\newtheorem{remark}[theorem]{Remark}

\newcommand{\QQ}{\mathbb{Q}}             %% rational numbers
\newcommand{\RR}{\mathbb{R}}             %% real numbers
\newcommand{\Pol}{\mathcal{P}}
\newcommand{\NN}{\mathbb{N}}             %% natural numbers
\newcommand{\ZZ}{\mathbb{Z}}             %% integers
\newcommand{\xx}{{\mathbf{x}}}

\newcommand{\bb}{{\mathbf{b}}}

\newcommand{\hot}{\text{lex. gr.}}
\newcommand{\rank}{\operatorname{rank}}
\begin{document}

\title{Vector partition functions and Kronecker coefficients}
\author{Marni Mishna, Mercedes Rosas, Sheila Sundaram}
\date{\today}
\maketitle

\begin{abstract}
The Kronecker coefficients are the structure constants for the restriction of irreducible
representations of the general linear group $GL(n m)$ into irreducibles for the subgroup 
$GL(n)\times GL(m)$.
In this work we study the  quasipolynomial nature of the Kronecker function using elementary tools from polyhedral geometry. 
We  write the Kronecker function in terms of coefficients of a vector partition function. This allows us to define a new family of coefficients, the atomic Kronecker coefficients.
Our derivation is  explicit and self-contained, and gives a new exact formula and an upper bound for the Kronecker coefficients in the first nontrivial case.   
\noindent{\bf 2010 Subject classification:} 05E10, 05E05, 17B10
%, 52B20
\end{abstract}

%{\small
%    \renewcommand*\contentsname{\small \sc Table of contents}
% \tableofcontents
%}

%\pagebreak
\section{Introduction} 

The \textit{subgroup restriction} or \textit{branching} problem investigates how an irreducible representation of a group $G$ decomposes into irreducibles  when restricted to a subgroup $H.$
 In this article, we study this branching for 
$GL(n)\times GL(m) $  viewed as a subgroup of $GL(n  m)$ via
the tensor product of matrices. The Kronecker coefficients are the structure constants for this branching.
 They are also important from a physicist's point of view.
Christandl, Harrow and Graeme have shown the relevance of Kronecker coefficients in the  study of the spectra of  bipartite quantum states with  two fixed marginal states, and studied the implications of their findings in quantum information theory, \cite{Christandl:Harrow:Mitchison,  Christandl:Mitchison, Christandl-PhD, CDW}. 
 The restriction of $GL(4)$ to $GL(2)\times GL(2)$ is of interest in nuclear physics where it is called Wigner supermultiplet theory, \cite{Wigner}.

Irreducible representations of $GL(n)$ are indexed by partitions
of length at most $n$. Therefore, the Kronecker coefficients are indexed by
triples of partitions of the same weight, and bounded lengths. We denote by
$g_{\mu,\nu, \lambda}$ the {\em  Kronecker coefficient indexed by 
 $\mu,\nu, \lambda.$ }
Alternatively, the Kronecker coefficients $g_{\mu,\nu,\lambda}$ can be defined by the expression:
\begin{align}\label{ourBranching}
s_{\lambda}[XY]=s_{\lambda}(x_1y_1, x_1y_2\cdots,x_{n}y_{m})= \sum_{\mu, \nu} g_{\mu,\nu,\lambda} \, s_\mu[X] s_\nu[Y]
\end{align}
where all partitions appearing in the equation have the same weight, and lengths bounded by $n, m $ and $nm$ respectively, and the $s_\lambda$'s denote the Schur polynomials.

 The \emph{Kronecker function} $\kappa_{m,n,l}$ is  a function  defined on triples of partitions $(\mu,  \nu, \lambda)$
of lengths bounded by $n, m $ and $l$ respectively, by
\begin{equation}
\kappa_{n,m,l} (\mu,  \nu, \lambda) = \kappa_{n,m,l} (\mu_1, \ldots, \mu_n, \nu_1,\ldots, \nu_m, \lambda_1, \ldots, \lambda_l): = g_{\mu, \nu,\lambda}.
\end{equation}

In this work, we use the restriction of the $Gl(nm)$-irreducible indexed by $\lambda$ to the subgroup $Gl(n)\times Gl(m)$ described  by Eqn.~(\ref{ourBranching})  to compute of $\kappa_{n,m,l}$.
 We introduce a new family of coefficients (also indexed by triples of partitions) that we call the \textit{atomic Kronecker coefficients}. They are defined 
 by a single vector partition function. As a result, they count integer points in polytopes, satisfy
 the saturation hypothesis  \cite{Kirillov:saturation, BOR-CC}, and are described by a  piecewise quasipolynomial.  We then show how to compute the actual Kronecker coefficients from these atomic coefficients.   We also show that, for partitions of lengths $2$, $2$ and $4$, the atomic Kronecker coefficients are an upper bound for the Kronecker coefficients.
 
The atomic Kronecker coefficients share many properties with the {\em reduced Kronecker coefficients} (a family of coefficients lying  between the Littlewood-Richardson coefficients and the Kronecker coefficients,  defined in Section \ref{section:reduced}):
 They
 contain enough information to compute from them  the value of any Kronecker       coefficient as an alternating sum.   Sometimes, the atomic Kronecker coefficients coincide with the Kronecker coefficients.  
However, the atomic Kronecker coefficients have a major advantage over the reduced Kronecker coefficients (that they share with the Littlewood--Richardson coefficients):  They satisfy the saturation hypothesis, whereas the reduced Kronecker coefficients do not \cite{PakPanova2020}.

In this paper we provide the theory and framework for computing $\kappa_{n,m,nm}$, although we focus mainly on the Kronecker function $\kappa_{2,2,4}$. In subsequent work (with Stefan Trandafir) we will report on an explicit implementation of the techniques developed in this paper   to  compute the Kronecker function $ \kappa_{2,3, 6}$.

The structure of the paper is the following. We begin in Section~\ref{sec-Polytopes} with a basic survey on polytopes and quasipolynomials. This is sufficient to understand the mechanics of our strategy. 
%The examples appearing in this section will play a major role in Section~\ref{sec-F22}. 
Our key idea is to use Eqn.~(\ref{ourBranching}) in conjunction with Cauchy's definition for Schur functions as a  quotient of alternants (equivalently, the Weyl character formula for the root system $A_n$) to make explicit the  relation between Kronecker coefficients and points in a polytope.

In Section~\ref{sec-F22}, we study  the smallest nontrivial example,  in which two of the partitions have length $\le 2$. We  provide concrete visualizations of the Kronecker functions~$\kappa_{2,2,2}$  and~$\kappa_{2,2,4}$. This can be made really explicit because the corresponding polyhedra are  of dimension 1 and 2 respectively. 
We give a new explicit closed form (Theorem~\ref{7termF22Alternant}) for the Kronecker coefficients in terms of coefficients of a  vector partition function $F_{2,2},$  as well as in terms of vector partition functions.  Our formula identifies 7 terms (out of a possible  total of 24) in the numerator of the Weyl character formula for the Weyl group $\mathbb{S}_4$, as the only  terms contributing to the Kronecker coefficient. 
 The number of chambers is large, even in the case 2-2-4, where it  was determined to be 74 in [BOR09b, BOR09a]. 
 We show that, in the case of Kronecker coefficients indexed by triples of partitions of length at most 2, 2, and 4, the atomic Kronecker coefficients  give an upper bound for the value of the Kronecker coefficients  (Theorem~\ref{thm:AtomicIsMax}). 
 We study the relative positions of the nonzero Littlewood--Richardson coefficients,  reduced Kronecker coefficients, and  atomic Kronecker coefficients inside the Kronecker cone (the polyhedral  cone generated by the nonzero Kronecker coefficients).
We study the dilated Kronecker coefficients,~$g_{k\lambda, k\mu,k\nu}$, defined for fixed $\lambda, \mu,$ and $\nu$, and $k \in \mathbb{N}$. We express these as a subseries of vector partition generating functions which implies that these are given by quasipolynomials in~$k$.  We show how to use Theorem~\ref{7termF22Alternant} to compute the dilated Kronecker coefficients $g_{k\mu, k\nu, k\lambda}$ in the 2-2-4 case. 

In Section~\ref{sec-Fnm} we consider the general situation.
Theorem~\ref{thm:AlwaysVecPart} presents an elementary but  nontrivial change of variables which converts the quotient of alternants  into a form recognizable as a  vector partition function, which we call $F_{n,m}$.  This facilitates our analysis since it returns us to the realm of Taylor series.  The function $F_{n,m}$ is the generating function of the atomic Kronecker coefficients.

The piecewise quasipolynomial nature of the Kronecker function has been the focus of much interest. The piecewise quasipolynomiality  follows from the work of Meinrenken and  Sjamaar \cite{qr0}. Both Christandl, Doran, and Walter~\cite{CDW} and Baldoni, Vergne, and Walter~\cite{BaldoniVergneWalter}
 describe and implement algorithms to compute the Kronecker coefficients. 
 Pak and Panova obtained  an interesting upper bound for the complexity of the  calculation of the Kronecker coefficients, see the proof of Lemma 5.4 in~\cite{Pak:Panova:complexityKC}.

%%%%%%%%%%%%%%%%%%%%%%%%%%%%%%%%%%%%%%%%%%%%%
\section{Polytopes and Quasipolynomials} 
\label{sec-Polytopes}

This section is a primer on polytope point enumeration and
quasipolynomiality. It can be skipped by those familiar with the topic. 
However, the examples we have chosen for this section are directly relevant in our study of the Kronecker coefficients.

%%%%%%%%%%%%%%%%%%%%%%%%%%%%%%%%%%%%%%%%%%%%%%%%%%%%%%%%%%%%%%%%%%%%%%
%%%%%%%%%%%%%%%%%%%%%%%%%%%%%%%%%%%%%%%%%%%%%%%%%%%%%%%%%%%%%%%%%%%%%%
%%%%%%%%%%%%%%%%%%%%%%%%%%%%%%%%%%%%%%%%%%%%%%%%%%%%%%%%%%%%%%%%%%%%%%

%\subsection{Polyhedra and polytopes} 
A \emph{polyhedron} $\Pol$ is
  the set of solutions of a (finite) system  and  inequalities:
\[
\Pol = \{ \xx \in \RR^d : A\xx \le \bb \},
\]
for a fixed matrix $A$ and vector $\bb$, where the ``$\le$" sign is to
be understood componentwise. 
A polyhedron is said to be \emph{rational} if both $A$ and $\bb$ have
integer entries.
A \emph{polytope} is a bounded polyhedron.
Note that any dilation of a polytope contains only a finite
number of integer points. 

The {\em  dimension} of a polytope is the dimension of  the affine space spanned by its vertices.
A \emph{$k$-simplex} is a $k$-dimensional polytope which is the convex
hull of $k+1$ vertices. 
 
A function $\phi : \NN \to \QQ$ is a 
 \emph{
    (one-variable) quasipolynomial} if there exist polynomials
  $p_0, p_1, \ldots, p_{k-1}$ in $\QQ[t]$ and a natural number $m>0$,
  a \emph{period\/} of $\phi$, such that
\[
\phi(t) = p_i(t), \text{ for } t \equiv i \mod m.
\]
The polynomials $p_i$ are the \emph{constituents} of $\phi$. The degree of a quasipolynomial is the maximum of the degrees of its components.

%%%%%%%%%%%%%%%%%%%%%%%%%%%%%%%%%%%%%%%%%%%%%%%%%%%%%%%%%%%%%%%%%%%%%%
%%%%%%%%%%%%%%%%%%%%%%%%%%%%%%%%%%%%%%%%%%%%%%%%%%%%%%%%%%%%%%%%%%%%%%
%%%%%%%%%%%%%%%%%%%%%%%%%%%%%%%%%%%%%%%%%%%%%%%%%%%%%%%%%%%%%%%%%%%%%%

\begin{example}\label{ZeroExample} 
Let  $\Pol$ be the one-dimensional polytope $[0, 1/2]$, and consider its integer dilations $ k \Pol = [0, k/2]$, as illustrated in Figure~\ref{fig:zeroexample}. 
 We want to count the number of integer points in the dilation of $\Pol$: 
\[
\phi_\Pol(k) := | \ZZ \cap k \Pol |
\] 
 Equivalently, we are interested in counting the number of nonnegative integer solutions to the inequality $0\le s_1\le k/2$. Then
\[
\phi_\Pol(k) = \left\lfloor \frac{k}{2} \right\rfloor+1 =
  \begin{cases}
    \frac{k+2}{2} &   \text{ if } k \equiv 0 \mod 2\\
  \frac{k+1}{2}     &  \text{ if } k  \equiv 1 \mod 2   \end{cases}
\]
is a linear quasipolynomial of period 2. 

\begin{figure}\center
\usetikzlibrary{arrows}
\begin{tikzpicture}
\draw[latex-] (-0.9,0) -- (2.4,0) ;
\draw[-latex] (-0.9,0) -- (2.4,0) ;
\foreach \x in  {-1/2, 0,1/2, 1, 3/2, 2}
\draw[shift={(\x,0)},color=black] (0pt,3pt) -- (0pt,-3pt);
\foreach \x in {0,1,2}
\draw[shift={(\x,0)},color=black] (0pt,0pt) -- (0pt,-3pt) node[below] 
{$\x$};
\draw[*-*] (-0.08,0) -- (0.58,0);
\draw[very thick] (0,0) -- (0.5,0);

\draw[*-*] (-0.08,0.2) -- (1.08,0.2);
\draw[very thick] (0,0.2) -- (1,0.2);

\draw[*-*] (-0.08,0.4) -- (1.58,0.4);
\draw[very thick] (0,0.4) -- (1.5,0.4);

\draw[*-*] (-0.08,0.6) -- (2.08,0.6);
\draw[very thick] (0,0.6) -- (2,0.6);
\end{tikzpicture}

\caption{The one dimensional polytope~$\mathcal{P}=[0, 1/2]$ and its first
  four integer dilations. The volume of the $k$-th dilation is $\left\lfloor \frac{k}{2} \right\rfloor+1$, a quasipolynomial in $k$.}
\label{fig:zeroexample}
\end{figure}
\end{example}

 A \emph{vector partition} of $\bb \in \NN^d$ is a way of decomposing $\bb$ as a  sum of nonzero vectors in $\NN^d$.  The order of the summands is irrelevant.
 We are interested in partitions whose \emph{parts} (nonzero summands) belong to a fixed finite sub-multiset $S$ of $ \NN^d$.  The \emph{vector partition function} $p_S : \NN^d \to \NN$ is the function that evaluated at $\bb$ gives the number of vector partitions of $\bb$ with parts in $S$.

Computing the value of the vector partition function $ p_S$ is  equivalent to finding the number of  nonnegative integer
 solution $\xx$ for the system of linear equation $A \xx= \bb$, where $A$ is
 the $d\times |S|$ matrix whose columns are the vectors in $S.$  
  It turns out that 
the matrices that appear in our work always contains a copy of the $n\times n$ identity matrix~$I_{n}$. %In this work all matrices  are always full rank.

%%%%%%%%%%%%%%%%%%%%%%%%%%%%%%%%%%%%%%%%%%%%%%%%%%%%%%%%%%%%%%%%%%%%%%
%%%%%%%%%%%%%%%%%%%%%%%%%%%%%%%%%%%%%%%%%%%%%%%%%%%%%%%%%%%%%%%%%%%%%%
%%%%%%%%%%%%%%%%%%%%%%%%%%%%%%%%%%%%%%%%%%%%%%%%%%%%%%%%%%%%%%%%%%%%%%

   Let $A$ be a $d\times n$  matrix with column vectors ${\bf a}_1,  {\bf a}_2, \ldots,  {\bf a}_n$. Let $pos(A)$ be the polyhedral cone generated by the columns of $A$. 
Given $\sigma \subseteq \{1,2,\ldots,n\}$, let $A_\sigma$ be 
the submatrix of $A$  consisting of those columns ${\bf a}_i$ with $i \in \sigma$. Let $\ZZ A_\sigma$ be the integral lattice spanned by the columns of $A_\sigma$.  A subset $\sigma$ is a \emph{basis} if $\rank(A)=\rank(A_\sigma)$. 

The \emph{chamber complex} is the polyhedral subdivision of the cone $pos(A)$ which is defined as the common refinement of the  cones $pos(A_\sigma)$, where $\sigma$ runs over all bases. 

 A function
 $g : \mathbb{Z}^n \to \mathbb{Q}$ is a {\em (multivariate) quasipolynomial} if there exists an $n$--dimensional lattice $\Lambda \subseteq \mathbb{Z}^n$, a set $\{ \mathbf{\lambda}_\mathbf{i} \}$ of coset representatives of $\mathbb{Z}^n/ \Lambda$, 
  and polynomials $p_i \in \mathbb{Q}[\mathbf{t}]$ such that $g(\mathbf{t}) = p_i(\mathbf{t})$, for $\mathbf{t} \in \mathbf{\lambda}_\mathbf{i} + \Lambda$. 
  
Blakley~\cite{Blakley} and Sturmfels~\cite{Sturmfels} have shown that 
\label{Blakley} 
  there exists a finite
  decomposition of $pos(A)$ (a chamber complex)  into rational polyhedral cones (chambers) such that in each chamber $\sigma$ the
  vector partition function $p_A(\bb) $ is given by a single
  multivariable quasipolynomial of degree $n-\rank(A)$. 
  
Moreover, the quasipolynomial $p_\sigma$ corresponding to chamber $\sigma$ counts the number of integral points in a    $(n-\rank(A))$ dimensional polytope $K_\sigma$. The
 leading term  of  $p_\sigma$ is always a polynomial. It gives the  volume of $K_\sigma$.
 
\begin{example}\label{ThirdExample}\label{F22}
Let $p_S(n,m)$ count the number of vector partitions of $\bb=(n,m)$ with
parts in $S=\{ (1,0),(0,1),(1,1),(1,2)\}$. Equivalently,  this is the number of nonnegative integer solutions $\xx$ to the system $A \xx= \mathbf{b}$, where
\[
A=
\begin{bmatrix}
 1 & 0 &  1  & 1 \\
0 & 1  & 1   &2   \\
\end{bmatrix} \qquad \mathbf{b}=\begin{bmatrix}n\\m\end{bmatrix}
\]
To determine one such partition, it suffices to determine the number
of parts equal to $(1,1)$ and $(1,2)$ in it. The standard basis vectors
serve as slack variables here,  consequently, the multiplicities of $(1,1)$ and $(1,2)$ should fulfill the  inequalities: 
\begin{align}\label{polytope22}
  \begin{cases}
 x_3+x_4 \le n\\
 x_3+2x_4 \le m
    \end{cases}
\end{align}
Therefore, the vector partition function $p_S$ counts nonnegative integer points in the polytope defined by the inequalities~\eqref{polytope22}. 

\begin{figure}
\begin{minipage}{.32\textwidth}\center
\usetikzlibrary{arrows}
\begin{tikzpicture}[scale=.4]
{ 
\draw[step=1.0,draw=gray!60,thin] (-0.5,-0.5) grid (8.5,8.5);

\draw[very thick] (0,8.2) -- (0,0) -- (8.2,0);

\node at (8.9,0) {$x_3$}; 
\node at (0,8.9) {$x_4$}; 
\fill[opacity=.4] (6,0) -- (0,0) -- (0,3);
\node at (6,-0.5) {$m$}; 
\node at (-0.5,3) {$\frac{m}{2}$}; 
\draw(0,7) -- (7,0);
\draw(0,3) -- (6,0);
\node at (7,-0.5) {$n$}; 
\node at (-0.5,7) {$n$}; 

}\end{tikzpicture}
\end{minipage}
\begin{minipage}{.32\textwidth}\center
\usetikzlibrary{arrows}
\begin{tikzpicture}[scale=.4]
{ 
\draw[step=1.0,draw=gray!60,thin] (-0.5,-0.5) grid (8.5,8.5);

\draw[very thick] (0,8.2) -- (0,0) -- (8.2,0);

\node at (8.9,0) {$x_3$}; 
\node at (0,8.9) {$x_4$}; 
\fill[opacity=.4] (3,0) -- (0,0) -- (0,3);
\node at (8,-0.5) {$m$}; 
\node at (-0.5,4) {$\frac{m}{2}$}; 
\draw(0,4) -- (8,0);
\draw(0,3) -- (3,0);
\node at (3,-0.5) {$n$}; 
\node at (-0.5,3) {$n$}; 
}\end{tikzpicture}
\end{minipage}
\begin{minipage}{.32\textwidth}\center
\usetikzlibrary{arrows}
\begin{tikzpicture}[scale=.4]
{ 
\draw[step=1.0,draw=gray!60,thin] (-0.5,-0.5) grid (8.5,8.5);
\fill[opacity=.4] (6,0) -- (0,0) -- (0,4) -- (4,2);
\draw[very thick] (0,8.2) -- (0,0) -- (8.2,0);

\node at (8.9,0) {$x_3$}; 
\node at (0,8.9) {$x_4$}; 

\node at (8,-0.5) {$m$}; 
\node at (-0.5,4) {$\frac{m}{2}$}; 
\draw(0,4) -- (8,0);
\draw(0,6) -- (6,0);
\node at (6,-0.5) {$n$}; 
\node at (-0.5,6) {$n$}; 

}\end{tikzpicture}
\end{minipage}
\caption{Three possibilities for the polytope defined by  the inequalities in Eqn.~\eqref{polytope22}}
\label{fig:polytopes}
\end{figure}

\begin{enumerate}

\item[(I)]
If $m \le n$ the first equation is redundant. We are counting integer points in the 2-simplex  defined by $x_3 \ge 0$, $x_4 \ge 0$, and $x_3+2 x_4 \le m$.

\item[(II)] If $n \le \frac{m}{2}$, it is the second equation that is redundant. We are counting integer points in the standard  2-simplex defined by  $x_3 \ge 0$, $x_4 \ge 0$, and $x_3+ x_4 \le n$.

\item[(III)] Finally, if $\frac{m}{2} < n < m,$ both inequalities are relevant. We are counting the number of points in the polytope with vertices 
$(0,0),(n,0),(0,\frac{m}{2}),$ $ (2n-m,m-n)$. We need to multiply by $2$ to get integer vertices, so the resulting quasipolynomial has  period $ 2$ for $m$. 

\end{enumerate}

The resulting piecewise quasipolynomial is then:\\

\begin{tabular}{cll}
  Region& &$p_S(n,m)$\\ %\midrule
  \hline
  I& $m\leq n$ & $\frac{m^2}{4}+ m + \frac{7}{8} + \frac{(-1)^m}{8}$\\
  II & $2n\leq m$& $\frac{n^2}{2}+\frac{3n}{2}+1$\\
  III  &$ n\leq m \leq 2n$ & $nm - \frac{n^2}{2} -\frac{m^2}{4} +\frac{n+m}{2} + \frac{7}{8} + \frac{(-1)^m}{8}$ \\%\bottomrule
 % \hline
 \end{tabular}\\

\begin{figure}
\begin{minipage}{.35\textwidth}\center
\usetikzlibrary{arrows}
\begin{tikzpicture}[scale=.4]
{ 
\draw[step=1.0,draw=gray!60,thin] (-0.5,-0.5) grid (8.5,8.5);
\fill[opacity=.4] (8.2,8.2) -- (0,0) -- (4.1,8.2);
\fill[opacity=.6] (8.2,8.2) -- (0,0) -- (8.2,0);
\fill[opacity=.2] (0,8.2) -- (0,0) -- (4.1,8.2);
\draw[very thick] (0,0) -- (8.2,8.2);
\draw[very thick] (0,0) -- (4.1,8.2);
\draw[very thick] (0,8.2) -- (0,0) -- (8.2,0);
\foreach \i in {0,...,8}
 { \foreach \j in {0,...,8}
   { \filldraw(\i,\j) circle (1pt);   
   }
 }
 \node[draw,color=white] at (1,5) {\bf II}; 
\node[draw, color=white] at (5,7) {\bf III}; 
\node[draw, color=white] at (5,2) {\bf I}; 
\node at (8.9,0) {$n$}; 
\node at (0,8.9) {$m$}; 
}\end{tikzpicture}
\end{minipage}\\
\caption{The chambers giving the value of $p_S(n,m)$, the number of vector partitions of $(n,m)$ with parts in $S=\{(1,0), (0,1), (1,1), (1,2) \}$.}
\label{fig:chambers}
\end{figure}

\smallskip
\noindent
where the three different regions are illustrated in Figure~\ref{fig:chambers}.

\end{example}

A function $p_A(\bb) $ that satisfies the conclusions of the  theorem of Blakley and  Sturmfels is known as a  {\em 
piecewise quasipolynomial function}. Vector partition functions are piecewise quasipolynomials. However,  a piecewise quasipolynomial
 need not count integral points in polytopes.

\section{A vector partition function for Kronecker coefficients}
\label{sec-F22}

Let $\lambda=(\lambda_1, \lambda_2, \ldots, \lambda_n)$ be  a partition, and let $\delta_n=(n-1,\cdots, 1, 0)$. 
The alternant $a_\lambda$  is defined as
$
a_\lambda(x_1, x_2, \ldots, x_n)=\det(x_i^{\lambda_j})_{i,j=1}^n.
$
Cauchy defined Schur polynomials in terms of the alternant as follows
\[
s_\lambda[X]=\frac{a_{\lambda+\delta}(x_1, x_2, \ldots, x_n)}{a_{\delta}(x_1, x_2, \ldots, x_n)}=\frac{\det(x_i^{\lambda_j+n-j})_{i,j=1}^n}{\prod_{1\le i<j \le n}(x_i-x_j)}.
\]
Let $\mu$, $\nu$, and $\lambda$ be three partitions of the same weight satisfying that $\ell(\mu) \le n$, $\ell(\nu) \le m$  and $\ell(\lambda) \le  nm$.

Given $X=\{x_1,\ldots,x_{n}\}$, $Y=\{y_1,\ldots,y_{m}\}$, define 
$
s_{\lambda}[XY]$ as $s_{\lambda}(x_1y_1, x_1y_2\cdots,x_{n}y_{m}).
$
This is a symmetric function in the $x$'s and the $y$'s separately. Since Schur functions form an integral basis for the algebra of symmetric functions, we can write  \begin{align}\label{comultiplication}
s_{\lambda}[XY]= \sum_{\mu, \nu} g_{\mu,\nu,\lambda} \, s_\mu[X] s_\nu[Y].
\end{align}
The   coefficients $g_{\mu,\nu,\lambda}$ 
are the \emph{Kronecker coefficients.} Formula ~\eqref{comultiplication} ishows that if  $\ell(\mu) \le n$, $\ell(\nu) \le m$, then 
$g_{\mu,\nu,\lambda}$  is nonzero only if $\ell(\lambda) \le  nm$.

Combining Cauchy's definition of a Schur polynomial  with the comultiplication formula \eqref{comultiplication}, we obtain the following  identity 
\begin{equation}\label{main}
\frac{{a_{\delta_n}[X]}{a_{\delta_m}[Y]}}{a_{\delta_{nm}}[XY]}   \,\,   a_{\lambda+\delta_{nm}}[XY]= \sum_{\mu, \nu} g_{\mu,\nu,\lambda} \,  a_{\mu+\delta_n}[X] a_{\nu+\delta_m}[Y].
\end{equation}

In the preceding formula, the right hand side is finite: The alternant $a_{\delta_{nm}}[XY]$ is a polynomial  consisting of $(nm)!$ nonzero monomials. The factor $\frac{{a_{\delta_n}[X]}{a_{\delta_m}[Y]}}{a_{\delta_{nm}}[XY]} $ on the left is a rational function. The numerator divides the denominator, and the left hand side simplifies to an expression of the form $ a_{\lambda+\delta_{nm}}[XY]$ divided by a polynomial. We will take a closer look at this in the next section.

\subsection{Kronecker coefficients for triples  of lengths at most $ 2,2,4$}

Our approach is best illustrated   in the simplest nontrivial case of three partitions $\mu, \nu, \lambda$ with lengths at most $2, 2, $ and $4$, respectively.
We dedicate the rest of Section \ref{sec-F22} to this particular case.

Let $\lambda$ be a fixed partition of length $\le 4$. Since
Schur functions are homogeneous polynomials, without loss of information, we can set $X=\{1, x\}$, $Y=\{1,y\}$, $XY=\{1,x,y,xy\}$ in Eq.~\eqref{main}. We obtain
\begin{equation}\label{KroneckerCauchy}
\frac{{a_{\delta_2}[X]}{a_{\delta_2}[Y]}}{a_{\delta_{4}}[XY]} \,\,   a_{\lambda+\delta_{4}}[XY]=
	 \sum_{\mu, \nu}  g_{\mu,\nu,\lambda} \,  x^{\mu_2}y^{\nu_2} + \hot
\end{equation}
where \lq\lq\hot "  \ stands for \lq\lq lexicographically greater terms."  The form arises since $\nu$, $\mu$ and $\lambda$ are partitions of the same integer $N$, and furthermore $\mu_2,\nu_2\le N/2$.

The quotient on the left simplifies drastically, into a simple rational function:
\[\frac{{a_{\delta_2}[X]}{a_{\delta_2}[Y]}}{a_{\delta_{4}}[XY]} = \frac{1}{{ x^2 y}\,\,(1-y/x)(1-xy)(1-x)(1-y)}.\]
We set 
\begin{equation}
\label{atomic1stdef}
\bar{F}_{2,2}(x,y):=\frac{1}{(1-y/x)(1-xy)(1-x)(1-y)}, 
\end{equation} and note that we can determine an iterated Laurent series expansion of this rational function by first developing a series expansion in $y$, and then in $x$:
\begin{equation}\label{F22series}\small
\bar{F}_{2,2}(x,y)= 
1+x+x^2+\dots + y(x^{-1}+2+\dots) + y^2(x^{-2}+2x^{-1}+\dots) + O(y^3),
\end{equation}
valid in a nonempty polydisc defined by $0<|xy|<|y|<|x|<1$.
This expansion is fundamental to our next step. 

This reduces the computation of Kronecker coefficients to straightforward series manipulations since $a_{\lambda+\delta_4}[XY]$ is just a polynomial. Restating Eqn.~\eqref{KroneckerCauchy}, we can use the series $\bar{F}_{2,2}(x,y)$ to compute the Kronecker coefficients:
\begin{align}\label{Kronecker1stdef}
\sum_{\mu, \nu}   g_{\mu,\nu,\lambda} x^{\mu_2} y^{\nu_2}
+ \hot
&=
\dfrac{a_{\lambda+\delta_4}[XY]}{(1-y/x)(1-xy)(1-x)(1-y)x^2y}\\
&=a_{\lambda+\delta_4}[XY]\dfrac{\bar{F}_{2,2}(x,y)}{x^2y}.\nonumber
\end{align}

 We can approximate the Kronecker coefficients via the following modification: We replace $a_{\lambda+\delta_{4}}[XY]$ with a single term. The resulting rational function no longer has a finite series expansion, but we can manipulate its iterated Laurent series expansion. Specifically, we replace the alternant $a_{\delta_{4}+\lambda}[XY]$ by its lexicographically least  monomial, which we denote by $S( a_{\lambda+\delta_{4}}[XY])$. This is explicitly computable by analysis of the determinant computation, as it is the product of the terms along the main diagonal:
\[S( a_{\lambda+\delta_{4}}[XY])=
(xy)^{\lambda_4}\cdot y^{\lambda_3+1} \cdot x^{\lambda_2+2} ={ x^2y}\cdot x^{\lambda_4+\lambda_2} y^{\lambda_4+\lambda_3}.\]

We name the coefficients in the resulting Laurent series expansion as follows:

\[
\frac{{a_{\delta_2}[X]}{a_{\delta_2}[Y]}}{a_{\delta_{4}}[XY]} \,\,  S( a_{\lambda+\delta_{4}}[XY])
	 =\sum_{\stackrel{b>\lambda_4+\lambda_3+1}{b\geq -a+2+\lambda_3+\lambda_4}} \tilde g_{(|\lambda|-a,a),(|\lambda|-b, b),\lambda} \,  x^{a}y^{b}. 
\]
The coefficients indexed by actual partitions turn out to have the feature that they are easy to compute, and in some circumstances are reasonable approximations their actual Kronecker coefficient analogues.  However, before we study them further, we use a change of basis to make the problem more combinatorial. 

\smallskip
\begin{lemma}
After the change of basis  $x=s_1$ and $y=s_0s_1,$  $\bar{F}_{2,2}(x,y)$ becomes a vector partition function:
 \begin{align}\label{defF22}
  \nonumber
 F_{2,2}(s_0, s_1)&=\sum \tilde g_{\mu,\nu,\lambda} s_0^{\nu_2-\lambda_3-\lambda_4} s_1^{\mu_2+\nu_2-\lambda_2-\lambda_3-2\lambda_4}\\
 &=\frac{1}{(1-s_0)(1-s_1)(1-s_0s_1)(1-s_0s_1^2)}.
 \end{align}
 \end{lemma}
This change of variables is desirable as it returns us to the realm of Taylor series.  Note that the assumptions $0<|xy|< |y|<|x|<1$, that define the domain of convergence of this series translate to $0 < |s_0|, |s_1|  < 1 $. This is  the vector partition function of Example~\ref{F22}.

The following observation is useful in the present discussion. We respect our coefficient ordering by noting that when we say the coefficient of $x^i y^j$ in $\bar{F}_{2,2}(x,y)$, denoted $[x^iy^j]\bar{F}_{2,2}(x,y)$ we mean $[y^j][x^i]\bar{F}_{2,2}(x,y)$ since our series expansion prioritizes $x$. The order is interchangeable in extractions of $F_{2,2}(s_0,s_1)$, since it is a finite product of Taylor (geometric) series. In our analysis of the coefficients, we choose between the original series $\bar{F}_{2,2}(x,y)$ in the variables $x,y,$ and the vector partition function $F_{2,2}(s_0,s_1)$ depending upon which is more convenient.  More precisely, we use Eqn.~\eqref{defF22} to apply techniques from  the theory of vector partition functions and polyhedral geometry.  However Eqn.~\eqref{Kronecker1stdef} is the more natural choice to analyse the alternant $a_{\lambda+\delta_4}[XY].$ 
 \begin{proposition}\label{barF22}  The coefficient of $x^i y^j$ in $\bar{F}_{2,2}(x,y)$, which is also the coefficient of $s_0^j s_1^{i+j}$ in the Taylor expansion of $F_{2,2}(s_0,s_1)$,  is nonzero if and only if $j\geq 0$ and $i+j\geq 0.$ 
 \end{proposition}
  It follows immediately from this or from Eqn.~\eqref{atomic1stdef} and Section~\ref{F22}, that

 \begin{proposition}\label{ineqatomiccone22} 
 The atomic Kronecker coefficient  $\tilde g_{\mu, \nu, \lambda}$ is nonzero if and  only if 
\begin{align}\label{Bravyi}
\begin{cases}
&  \lambda_2+\lambda_3+2\lambda_4 \le \mu_2 + \nu_2 \\
& \lambda_3+\lambda_4 \le \nu_2.
\end{cases}
\end{align}
Moreover, the value of $\tilde{g}_{\mu, \nu, \lambda}$ is given by a quadratic quasipolynomial:
\[\tilde g_{\mu, \nu, \lambda}= p_S(\nu_2-(\lambda_3+\lambda_4),  \mu_2+\nu_2-
(\lambda_2+\lambda_4 )-(\lambda_3+\lambda_4)),\]  %with respect to the second argument.
where $p_S$ is the vector partition function of Example~\ref{F22}.

\end{proposition}
These two inequalities have been previously derived  by Bravyi in the context of quantum physics, \cite{Bravyi}. We will refer to them as the {\em first and second inequalities of Bravyi.}

\begin{corollary} The value of the atomic Kronecker coefficient  depends only on the values of the two linear forms $\nu_2-(\lambda_3+\lambda_4)$ and $  \mu_2+\nu_2-
(\lambda_2+\lambda_4 )-(\lambda_3+\lambda_4)$. 
Furthermore, when either one is equal to zero (and the other nonnegative), the corresponding atomic Kronecker coefficient is equal to 1.

\end{corollary}

\subsection{From $\bar F_{2,2}$ to an exact expression for Kronecker coefficients}
The rational series $F_{2,2}$ is not directly the generating series for the Kronecker coefficients because we truncated some polynomials in its construction. The main result of this section, Theorem~\ref{thm:MAIN1}, is an exact formula for the Kronecker coefficients in the $n=m=2$ case. 

 Since $\ell(\lambda)\le 4,$ the number of terms  in the  expansion of the alternant $a_{\lambda+\delta_4}[XY]$ is $4!=24$. However, it turns out that only \textit{seven} terms contribute to the Kronecker coefficient.
The following theorem explicitly identifies which terms of the alternant contribute to the Kronecker coefficient.
The polynomial in Theorem~\ref{7termF22Alternant} is minimal:  Example~\ref{Ex7terms} exhibits a combination in which all seven terms contribute nontrivially to the Kronecker coefficient, with NO  cancellation between any pairs of terms.

 \begin{theorem}\label{7termF22Alternant}\label{thm:MAIN1}
Assume $\ell(\lambda)\leq 4,$ and $\ell(\mu), \ell(\nu)\leq 2.$ Also assume $\mu_2\geq \nu_2.$ Then the  Kronecker coefficient $g_{\mu, \nu, \lambda}$ is equal to each of the following:
\begin{enumerate}
\item the coefficient of $x^{\mu_2} y^{\nu_2}$ in 
$ P_\lambda(x,y)\bar{F}_{2,2}(x,y)$, 
 where $P_\lambda(x,y)$ is the polynomial consisting of the following seven terms:
 \begin{multline}\label{explicitF22extraction}
 y^{\lambda_3+\lambda_4}  (x^{\lambda_2+\lambda_4} 
 -x^{\lambda_2+\lambda_3+1}-x^{\lambda_1+\lambda_4+1}
 +x^{\lambda_1+\lambda_3+2  })\\
 +y^{\lambda_2+\lambda_4+1}(-x^{\lambda_3+\lambda_4-1}+x^{\lambda_2+\lambda_3+1}
 +x^{\lambda_1+\lambda_4+1})
 \end{multline}
 A monomial $y^b x^a$ in $P_\lambda$ makes a nonzero contribution to $g_{\mu,\nu,\lambda}$ if and only if 
 $ b\leq \nu_2 \text{ and } b+a\leq \mu_2+\nu_2$. 
 \item the following 7-term linear combination of vector partition  functions $p_S(n,m)$:
{\tiny{ 
 \begin{multline}\label{Exact}g_{\mu,\nu,\lambda}=
 p_S(\nu_2-(\lambda_3+\lambda_4), \nu_2-(\lambda_3+\lambda_4)+ \mu_2-(\lambda_2+\lambda_4))\\
 -p_S(\nu_2-(\lambda_3+\lambda_4), \nu_2-(\lambda_3+\lambda_4)+ \mu_2-(\lambda_2+\lambda_3+1))\\
 -p_S(\nu_2-(\lambda_3+\lambda_4), \nu_2-(\lambda_3+\lambda_4)+ \mu_2-(\lambda_1+\lambda_4+1))\\
+ p_S(\nu_2-(\lambda_3+\lambda_4), \nu_2-(\lambda_3+\lambda_4)+ \mu_2-(\lambda_1+\lambda_3+2))\\
-p_S(\nu_2-(\lambda_2+\lambda_4+1), \nu_2-(\lambda_2+\lambda_4+1)+\mu_2-(\lambda_3+\lambda_4-1))\\
+p_S(\nu_2-(\lambda_2+\lambda_4+1), \nu_2-(\lambda_2+\lambda_4+1)+\mu_2-(\lambda_2+\lambda_3+1))\\
+p_S(\nu_2-(\lambda_2+\lambda_4+1), \nu_2-(\lambda_2+\lambda_4+1)+\mu_2-(\lambda_1+\lambda_4+1))
 \end{multline}}}
 \end{enumerate}
   \end{theorem}

\begin{example}\label{Ex7terms} (\textbf{Minimality of the polynomial in Theorem~\ref{7termF22Alternant}}) Let $\lambda=(12,7,4,1), \mu=\nu=(12,12).$ 
From Theorem~\ref{7termF22Alternant}, the Kronecker coefficient 
$g_{\mu, \nu, \lambda}$ is the coefficient of $x^{12}y^{12}$ in the product $P_\lambda\bar{F}_{2,2}(x,y),$  where 
%\[g_{\mu, \nu, \lambda}=[x^{12}y^{12}]P_\lambda\bar{F}_{2,2}(x,y).\] 
\[P_{\lambda}=y^5(x^8-x^{12}-x^{14}+x^{18}) +y^9(-x^{4}+x^{12}+x^{14}).\]
%Sheila2019-12-17 We apply the equivalence from Remark~\ref{remark:compute}:
Using Eq.~(\ref{Exact}), we have $g_{\mu,\nu,\lambda}$ is equal to
\begin{align*}
&p_S(7,11)-p_S(7,7)-p_S(7,5)+p_S(7,1)-p_S(3,11)+p_S(3,3)+p_S(3,1)\\
&=32-20-12+2-10+6+2=0.
\end{align*}

This example is noteworthy because the Kronecker coefficient vanishes, but there is no cancellation between pairs of the seven coefficients above. By definition, the atomic coefficient is the contribution from the first monomial $y^5x^8$ in the expansion of $P_\lambda$ above, (it is also the lexicographically least monomial),  hence~$\tilde{g}_{\mu, \nu, \lambda}=p_S(7,11)=32.$
\end{example}  

A similar example where $\lambda$ has only 3 parts follows.
\begin{example}\label{Ex7terms2} Let $\lambda=(11,5,2,0), \mu=\nu=(9,9).$ Using Eq.~(\ref{Exact}), we again have that $g_{\mu,\nu,\lambda}$ equals
\begin{multline*}
p_S(7,11)-p_S(7,8)-p_S(7,4)+p_S(7,1)-p_S(3,11)+p_S(3,4)+p_S(3,0)\\
=32-24-9+2-10+8+1=0.
\end{multline*}
\end{example}
 
When $\lambda$ has two parts, 
at most the first six terms in \eqref{Exact} can contribute to the Kronecker coefficient, since the seventh term is necessarily zero. 
This follows because the second argument to the seventh (and last)  vector partition function in ~\eqref{Exact} is negative, using the fact that $\nu_2\le\mu_2\le\mu_1$ :
\[\nu_2+\mu_2-(\lambda_2+\lambda_1+2)\le \mu_1+\mu_2 
-(|\lambda|+2)= -2.\]

 \begin{proof}[Proof of Theorem~\ref{7termF22Alternant}]  The  statement about which monomials $y^bx^a$ in $P_\lambda$ can make a nonzero contribution is a direct consequence of Proposition~\ref{barF22}. 
 Eqn.~\eqref{Exact} follows from the polynomial $P_\lambda$ because of the following observation:
 If $x^ay^b$ is a monomial in $P_\lambda$ making a nonzero contribution to $g_{\mu,\nu,\lambda}$, that value is 
 \[\pm[s_0^{\nu_2-b}s_1^{\mu_2+\nu_2-a-b}]F_{2,2}=\pm p_S(\nu_2-b,\mu_2+\nu_2-a-b).\] This follows from the substitution $x\mapsto s_1$ and $y\mapsto s_0s_1$, Proposition~\ref{barF22} and Section~\ref{F22}. 
  
 To prove the assertion in 
  ~\eqref{explicitF22extraction}, we manipulate the alternant directly.  The first part of the proof consists of a judicious choice in expanding the determinant, followed by a careful analysis of the resulting monomials. Twelve of these are easily  shown to make a contribution of zero.  The final reduction to 7 terms is more delicate; it will be useful to consult Figure~\ref{fig:chambers} in Section~\ref{sec-Polytopes}, since the precise formula for the quasipolynomial $p_S(n,m)$ in one specific chamber will play a crucial role in the proof.

We start by expanding the alternant $a_{\lambda+\delta_4}$  by the fourth column.  Writing $A_{i,j}$ for the minor of the entry in row $i$ and column $j,$ this gives 
 \[(xy)^{\lambda_4} A_{4,4}-(xy)^{\lambda_3+1} A_{3,4}
 +(xy)^{\lambda_2+2} A_{2,4}-(xy)^{\lambda_1+3} A_{1,4}.\]
 
 Each of these 3 by 3 minors will give 6 terms.  It follows from Eqn.~\eqref{Kronecker1stdef} that the contributions to the Kronecker coefficient are obtained by extracting the coefficient of $x^{\mu_2} y^{\nu_2}$ in the product of these monomials with $(x^{-2}y^{-1})\bar{F}_{2,2}.$ The four tables below, listed in the same order as the minors above, show the resulting monomials, \textit{with sign}, with the exponent of $x$ diminished by 2 and  the exponent of $y $ diminished by 1. For ease of reading we list only the exponents of $x$ and $y$. 

\begin{table}[h]
\begin{subtable}[t]{\textwidth}\small\center
\begin{tabular}{| c| c| c ||c| c| c| }\hline
   & Power of $x$  &Power of $y$ &  & Power of $x$  &Power of $y$\\
 \hline
  $(+)$ &$\lambda_4+\lambda_1+3-2 $&$\lambda_4+\lambda_2 +2-1$&$(-)$ & $\lambda_4+\lambda_2 +2-2$&$\lambda_4+\lambda_1+3-1$\\
 $(-)$ &$\lambda_4+\lambda_1+3-2 $ &$\lambda_4+\lambda_3+1-1$
 &$(+)$ &$\lambda_4+\lambda_3+1-2 $ &$\lambda_4+\lambda_1+3 -1$\\
 $(+)$ &$\lambda_4+\lambda_2 +2-2$  &$\lambda_4+\lambda_3+1-1$
 &$(-)$&$\lambda_4+\lambda_3+1-2$  &$\lambda_4+\lambda_2 +2-1$\\
\hline
 \end{tabular}
 \caption{Expansion of $(xy)^{\lambda_4} A_{4,4}x^{-2}y^{-1}$ }
 \end{subtable}
\begin{subtable}[t]{\textwidth}
\small\center
 \begin{tabular}{| c| c| c ||c| c| c| }\hline
  & Power of $x$  &Power of $y$ &  & Power of $x$  &Power of $y$\\
 \hline
$(-)$ &$\lambda_3+\lambda_1+4-2 $&$\lambda_3+\lambda_2 +3-1$&$(+)$ & $\lambda_3+\lambda_2 +3-2$&$\lambda_3+\lambda_1+4-1$\\
$(+)$ &$\lambda_3+\lambda_1+4-2 $ &$\lambda_4+\lambda_3+1-1$
&$(-)$ &$\lambda_4+\lambda_3+1-2 $ &$\lambda_3+\lambda_1+4-1 $\\
$(-)$ &$\lambda_3+\lambda_2 +3-2$  &$\lambda_4+\lambda_3+1-1$
&$(+)$&$\lambda_4+\lambda_3+1-2$  &$\lambda_3+\lambda_2 +3-1$\\
\hline
\end{tabular}
\caption{Expansion of $-(xy)^{\lambda_3+1} A_{3,4}x^{-2}y^{-1}$}
\end{subtable}
\begin{subtable}[t]{\textwidth}
\small\center
\begin{tabular}{| c| c| c ||c| c| c| }\hline
   & Power of $x$  &Power of $y$ &  & Power of $x$  &Power of $y$\\
 \hline
  $(+)$ &$\lambda_2+\lambda_1+5-2 $&$\lambda_2+\lambda_3 +3-1$&$(-)$ & $\lambda_2+\lambda_3 +3-2$&$\lambda_2+\lambda_1+5-1$\\
 $(-)$ &$\lambda_2+\lambda_1+5-2 $ &$\lambda_2+\lambda_4+2-1$
 &$(+)$ &$\lambda_2+\lambda_4+2-2 $ &$\lambda_2+\lambda_1+5-1 $\\
 $(+)$ &$\lambda_2+\lambda_3 +3-2$  &$\lambda_2+\lambda_4+2-1$
 &$(-)$&$\lambda_2+\lambda_4+2-2$  &$\lambda_2+\lambda_3 +3-1$\\
 \hline
\end{tabular}
\caption{Expansion of $(xy)^{\lambda_2+2} A_{2,4}x^{-2}y^{-1}$}
\end{subtable}
\begin{subtable}[t]{\textwidth}
\center\small
\begin{tabular}{| c| c| c ||c| c| c| }\hline
  & Power of $x$  &Power of $y$ &  & Power of $x$  &Power of $y$\\
 \hline
 $(-)$ &$\lambda_1+\lambda_2+5-2 $&$\lambda_1+\lambda_3 +4-1$&$(+)$ & $\lambda_1+\lambda_3 +4-2$&$\lambda_1+\lambda_2+5-1$\\
 $(+)$ &$\lambda_1+\lambda_2+5-2 $ &$\lambda_1+\lambda_4+3-1$
 &$(-)$ &$\lambda_1+\lambda_4+3-2$ &$\lambda_1+\lambda_2+5-1 $\\
 $(-)$ &$\lambda_1+\lambda_3 +4-2$  &$\lambda_1+\lambda_4+3-1$
 &$(+)$&$\lambda_1+\lambda_4+3-2$  &$\lambda_1+\lambda_3 +4-1$\\
 \hline
\end{tabular}
\caption{Expansion of $-(xy)^{\lambda_1+3} A_{1,4}x^{-2}y^{-1}$}
\end{subtable}
\caption{Terms appearing in the co-factor expansions}
\end{table}

 If $x^a y^b$ is a monomial in the tables, then by Proposition~\ref{barF22}, we  must have $\nu_2\geq b$ and $\mu_2+\nu_2\geq a+b.$ The latter condition immediately eliminates all 6 monomials in Table~1~(D), since the sum of exponents there clearly (strictly) exceeds 
 $\sum_{i=1}^4 \lambda_i,$ whereas $\nu_2, \mu_2\leq |\lambda|/2.$
  For the same reason the monomials in the first two lines of Table~1~(C), as well as the two monomials in the first row of Table~1~(B), are also eliminated.  We are left with the following 12 monomials, from which we will  eliminate  the five underlined  terms,  leaving the seven monomials in the expression   
  ~\eqref{explicitF22extraction}.
  \begin{multline}\label{firstpassF22extraction}
 y^{\lambda_3+\lambda_4}  (x^{\lambda_2+\lambda_4} 
 -x^{\lambda_2+\lambda_3+1}-x^{\lambda_1+\lambda_4+1}
 +x^{\lambda_1+\lambda_3+2  })\\
 +y^{\lambda_2+\lambda_4+1}(x^{\lambda_2+\lambda_3+1}
 +x^{\lambda_1+\lambda_4+1})\\
 +x^{\lambda_3+\lambda_4-1}(- y^{\lambda_2+\lambda_4+1}+\underline{y^{\lambda_2+\lambda_3+2}}+\underline{\underline{ y^{\lambda_1+\lambda_4+2} }}
 -\udensdash{$y^{\lambda_1+\lambda_3+3}$})\\
% -\underline{\underline{\underline{y^{\lambda_1+\lambda_3+3}}}})\\
 +x^{\lambda_2+\lambda_4} ( -\underline{y^{\lambda_2+\lambda_3+2} } - 
  \underline{\underline{ y^{\lambda_1+\lambda_4+2}}} )
 \end{multline}
 Examining  the term 
   \udensdash{$x^{\lambda_3+\lambda_4-1} y^{\lambda_1+\lambda_3+3}$}, 
 %$\underline{\underline{\underline{x^{\lambda_3+\lambda_4-1}y^{\lambda_1+\lambda_3+3}}}},$ 
 we see that this monomial \textit{cannot} contribute to the Kronecker coefficient, since we must have $\nu_2\geq \lambda_1+\lambda_3+3.$ 
 But this is impossible because it implies
 \[\sum_{i=1}^4\lambda_i=\nu_1+\nu_2\geq 2\nu_2\geq 2(\lambda_1+\lambda_3+3), \text{ i.e. }
  \lambda_2+\lambda_4\geq \lambda_1+\lambda_3+6.\]
 
 We  can also eliminate the remaining four underlined terms in ~\eqref{firstpassF22extraction}.
 % the two middle terms in the third line above, and the two terms in the fourth line.  
 First note the following crucial fact:  every monomial in the third and fourth  lines of ~\eqref{firstpassF22extraction} is of the form $x^ay^b$ where $a<b.$
 
 The coefficient of $x^{\mu_2}y^{\nu_2}$ in the product of each monomial $x^ay^b$ in the above polynomial with $\bar{F}_{2,2}(x,y)$ is equal to the coefficient of 
 $x^{\mu_2-a}y^{\nu_2-b}$ in $\bar{F}_{2,2}(x,y).$  
 % But the exponent of $x$ is not less than the exponent of $y$, because 
 %$\mu_2-a>\mu_2-b \geq (\nu_2-b).$ Hence by Proposition~\ref{barF22}, this coefficient is independent of the exponent of $x$ and  equals ${\nu_2-b+2\choose 2}$. 
 From Section~\ref{F22} and Proposition~\ref{barF22}, this coefficient equals $p_S(n_b, m_{b,a})$ where $n_b=\nu_2-b, m_{b,a}=(\nu_2-b )+(\mu_2-a ).$  Note that $a<b$ implies $\mu_2-a>\mu_2-b \geq (\nu_2-b),$ and hence $m_{b,a} > 2 n_b.$  It follows from (II) in Figure~\ref{fig:chambers} that $p_S(n_b, m_{b,a})=\binom{n_b+2}{2}$ is independent of $a,$ the exponent of $x.$ This holds for every term in lines 3 and 4 of ~\eqref{firstpassF22extraction}.
 
 Examining the two underlined middle terms of the third line, we see that each of the terms can be matched up with a monomial with the same $y$ exponent in the fourth line, to give $x^{a_1}y^b -x^{a_2} y^b.$ These correspond to extracting from $\bar{F}_{2,2},$  the coefficients  of $x^{\mu_2-a_1}y^{\nu_2-b}$ and $x^{\mu_2-a_2}y^{\nu_2-b}.$  
 
 Since $\mu_2\geq \nu_2,$ we have $\mu_2-a_i\geq \nu_2-a_i>\nu_2-b$ (recall that  $a_i<b$).   A necessary condition for either monomial to make a nonzero contribution is for the exponent $\nu_2-b$ of $y$  to be nonnegative. In the present situation, this forces the exponent $\mu_2-a_i$ of $x$ to be nonnegative as well. Hence either both monomials $x^{\mu_2-a_i}y^{\nu_2-b}, i=1,2,$ contribute to the Kronecker coefficient, or neither does.  Since (from the preceding paragraph), the contributions are independent of the $a_i$ and \textit{equal}  (to $\binom{\nu_2-b+2}{2}$), and the monomials come with opposite sign, their combined contribution is zero.
 \end{proof}
   
   We shall see in Theorem~\ref{thm:AtomicIsMax} that the monomials in the polynomial $P_\lambda$ have some rather remarkable properties.

 \begin{remark}
Calculations of $\kappa_{2,2,4}$ were previously explicitly worked out in~\cite{BriandOrellanaRosasquasi} using an identity describing the Kronecker coefficient as a linear combination of \textit{reduced} Kronecker coefficients (See Section \ref{RKC})\cite[Theorem 4]{BriandOrellanaRosasquasi}. Their approach differs from ours, but does permit determination that the number of chambers in the corresponding chamber complex is~$74$. This approach can also be compared with  \cite{Rosas-PhD} where a combinatorial interpretation for $\kappa_{2,2,4}$ was found as the difference of the number of integer points in two rectangles (mod 2).  
 \end{remark}

  \subsection{Examples}\label{SheilaExamples}
We illustrate these results with some examples. Since we invoke the computation of the coefficients $p_S(n,m)$ of $s_0^ns_1^m$ in $F_{2,2}(s_0, s_1)$ from Section~\ref{sec-Polytopes} extensively, we record the values of the following special coefficients:
\[p_S(n,0)\!=\!1\!=\!p_S(0,m);\ p_S(n,1)=2, n>0; \ p_S(1,m)=3, m\geq 2.
\]

\begin{example} Let $\lambda=(6,5,4,1), \mu=\nu=(9,7).$ 
Note that $\lambda_3+\lambda_4=5, \mu_2+\nu_2=14,$ 
and thus, by checking the inequalities \eqref{1stalternantInequality} and \eqref{2ndalternantInequality},
 we see that  only three of the seven terms from the polynomial $P_\lambda$  contribute to $g_{\mu,\nu,\lambda}, $ specifically those in the expansion of
\[y^{5}(x^{6}-x^{8}) +y^{7}(-x^{4} ).\]
We obtain the value
$g_{\mu,\nu,\lambda}
 =p_S(2,3)-p_S(2,1)-p_S(0,3)=2.$
% ATOMIC VALUE: 
% $< \tilde g_{\mu,\nu,\lambda}=p_S(2,3)=5.$
% %=b_{0,2}=5.
\end{example}

\begin{example}\label{ex:Atomic}
Our format is well suited to compute dilated Kronecker coefficients
(see  Section~\ref{Dilation}). Assume $k$ is a positive integer, and let $k\lambda=(6k,3k,2k), k\mu=(7k,4k)$ and $k\nu=(8k,3k)$. These are the dilations of the triple $((6,3,2), (7,4), (8,3))$.  The Kronecker coefficients are all atomic and given by the quasipolynomial $p_S(k, 2k) = (k+1)(k+2)/2$,  which counts integer points in dilations of the 2-simplex generated by $\{(0,0), (1,0), (0,1)   \}$. %This is precisely the polytope dilation in Figure~\ref{fig:simplex}.%, which is a visualization of the dilations of the corresponding polytope.
\end{example}

\subsection{Dilated Kronecker Coefficients}\label{Dilation}
 Fix $\mu,\nu$ and $\lambda$. The family of Kronecker coefficients given by the image of the function
 $k\mapsto g_{k\mu,k\nu,k\lambda}$, for $k=1, 2, \dots$ is a set of  \emph{dilated Kronecker coefficients},  and has been the center of a lot of attention.  When the lengths of the partitions $\mu,\nu,\lambda$ are bounded by 2,2, and 4, we can compute them using Theorem~\ref{thm:MAIN1},  and we can also write them as subseries of $F_{2,2}$ in a way that directly connects to vector partition functions. 

 \begin{example}[The Kronecker function does not count integer points in polytopes]\label{one-onecubed}
 Consider the dilated Kronecker coefficient $g_{(k,k), (k,k), (k,k)},$ for any positive integer $k$. 
 
By direct computation we see that only the first four monomials of $P_\lambda$ in Theorem~\ref{7termF22Alternant} contribute to the Kronecker coefficient. The formula of Section~\ref{F22} for $p_S(n,m)$ gives::
 \begin{align}
 g_{(k,k), (k,k), (k,k)} &=[x^ky^k](x^k-2x^{k+1}+x^{k+2})\bar{F}_{2,2}(x,y) \nonumber\\
 &=[s_0^k s_1^k] (1-2s_1+s_1^2)F_{2,2}(s_0, s_1) \label{eq:gkkk}\nonumber\\
 &=p_S(k,k)-2p_S(k,k-1)+p_S(k,k-2)\\
 &=\begin{cases} 1, &k \text{ even}\\
                                                       0, &k \text{ odd},
                                                       \end{cases}
 \end{align}

In striking contrast with the Littlewood-Richardson coefficients, the
Kronecker coefficients do not satisfy the saturation property:     There are holes in the Kronecker cone.
% sheila edit: 2018-11-23 and again 2018-11-26
Counterexamples and related conjectures can be found in~\cite{King-Sevilla, BOR-CC, Christandl-PhD}.  
Figure~\ref{fig:Holes} illustrates their location in a small, visualizable case. 
\begin{figure}\center

\includegraphics[width=.6\textwidth]{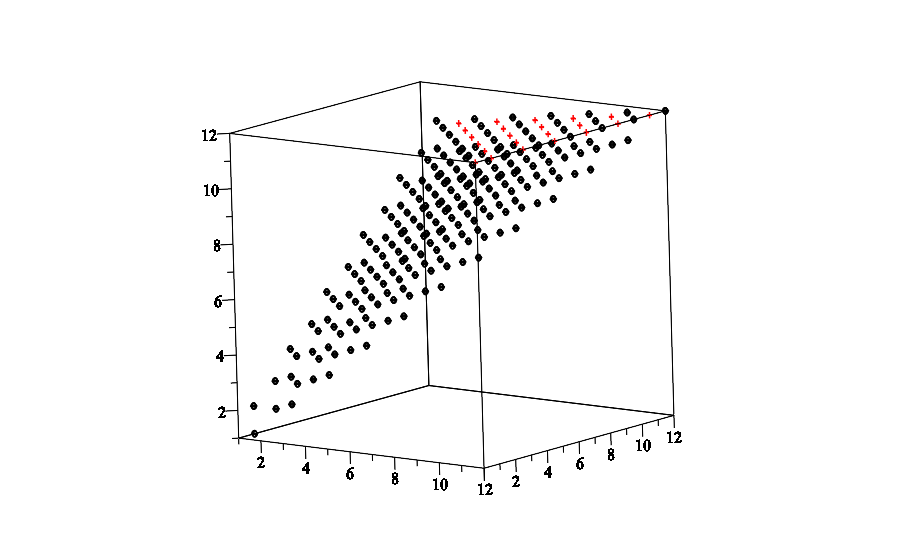}
\caption{Holes in the Kronecker cone when all three partitions are of length 2. The point at $(i,j,k)$ is black if $g_{(24-i, i) (24-j, j) (24-k, k)}$ is nonzero (assuming, $j\leq i\leq k\leq 24/2$). The points with red crosses, or no dots are 0. Note that the top face has both zero and nonzero values. These are the holes in the polytope.}
\label{fig:Holes}
\end{figure}

  The sequence $ g_{(k,k), (k,k), (k,k)}$, for $k\geq 0$ illustrates that the Kronecker coefficients cannot  possibly count points in the dilations of a polytope because such sequences are necessarily weakly increasing. That is, the Kronecker function does not count integer points in polytopes.  See~\cite{King-Sevilla, BOR-CC} for related results and conjectures. 
\begin{remark}[The holes of the Kronecker cone]
Example~\ref{one-onecubed}  illustrates the origin of the holes in the Kronecker cone  in Figure~\ref{fig:Holes}.  Note that the holes are all in
the face of the Kronecker cone defined by equations $\mu_1=\mu_2$, $\nu_1=\nu_2$, $\lambda_1=\lambda_2$, $\lambda_3=\lambda_4=0$.
It  is  always  the case that the zeros of the Kronecker cone are on its walls (facets) \cite{Manivel:asymptotics1}.

 This can also be seen for the example in Figure~\ref{fig:Holes} where the holes are all inside the face defined by $\lambda_1=\lambda_2$.
\end{remark}
\end{example}

\begin{example}  Consider an example of Baldoni and Vergne~\cite[Section 5.1.1]{BaldoniVergneWalter}, whose methods are quite different from ours. Let  $\lambda=(132,38,19,11),$  $\mu=(110,90),$ and $\nu=(120,80).$ We will compute an expression for the dilated Kronecker coefficient 
$g_{k\mu,k\nu,k\lambda}.$

We have $k(\mu_2+\nu_2)=170k,$ $k(\lambda_3+\lambda_4)=30k, $
$k(\lambda_2+\lambda_4)=49k,$ $ k\min(\lambda_2+\lambda_3, \lambda_1+\lambda_4)=k\min(57, 143)=57k.$ Eqn.~\eqref{Exact} of Theorem~\ref{7termF22Alternant} 
says that $g_{k\mu,\,k\nu,\,k\lambda}$ is equal to 
{\small
\[
p_S(50k, 91k)-p_S(50k,83k-1)
-p_S(31k-1, 91k-2)+p_S(31k-1,64k-2).\]}
From Section~\ref{F22}, the last two terms cancel each other because both correspond to Region II in Figure~\ref{fig:chambers}, %$m>2n$
and hence depend only on the first argument $n$ of $p_S(n,m).$  The two remaining terms  correspond to Region III.  The reader can check that using the formula for Region III  gives 
\[g_{k\mu,k\nu,k\lambda}=52k^2+\frac{25}{2}k+\frac{3}{4}+\frac{(-1)^k}{4},\]
in agreement with the result in \cite[Section 5.1.1]{BaldoniVergneWalter}.
\end{example}

 \subsection{Inequalities implying that a Kronecker coefficient is atomic}\label{appendix}
 
Analysing the order relations in the exponents appearing in Eqn.~\eqref{explicitF22extraction}  yields the following result.
 
 \begin{corollary}\label{atomicisKron} Assume $\mu_2 \ge \nu_2.$ The atomic coefficient equals the Kronecker coefficient if 
 \begin{enumerate}
 \item $\nu_2<\lambda_3+\lambda_4;$ 
 OR \item
 \begin{enumerate}
 \item $\lambda_3+\lambda_4\le \nu_2\leq \lambda_2+\lambda_4$ 
 and
 \item $(\lambda_3+\lambda_4)+(\lambda_2+\lambda_4) \leq 
 \mu_2+\nu_2\le 
 (\lambda_3+\lambda_4)+\min(\lambda_2+\lambda_3, \lambda_1+\lambda_4),$ 
 \end{enumerate}
 OR  \item $\lambda_2+\lambda_3+2\lambda_4 =
 \mu_2+\nu_2.$
 \end{enumerate}
 
  \end{corollary}

 \begin{proof} 
 We examine the exponents in ~\eqref{explicitF22extraction}. %
 Recall that a monomial $y^b x^a$ will make a nonzero contribution to 
 $g_{\mu,\nu,\lambda}$ if and only if 
$ b\leq \nu_2 \text{ and } b+a\leq \mu_2+\nu_2.$

In the first case both the atomic and Kronecker coefficient vanish.

We consider the second case.
The two exponents of $y$ in ~\eqref{explicitF22extraction} are ordered as follows:
 \begin{equation}\label{yInequality} \lambda_3+\lambda_4<\lambda_2+\lambda_4+1.\end{equation}

The five distinct exponents of $x$ occurring in \eqref{explicitF22extraction} satisfy 
 \begin{multline}\label{xInequality}\lambda_3+\lambda_4-1<\lambda_2+\lambda_4<\min(\lambda_2+\lambda_3+1, \lambda_1+\lambda_4+1)\\
 \leq \max(\lambda_2+\lambda_3+1, \lambda_1+\lambda_4+1)<\lambda_1+\lambda_3+2.
 \end{multline}
Compare with~\cite{Rosas-PhD}.  By considering the sequences of total degree  of the 7 monomials in each of the two lines of \eqref{explicitF22extraction},  we have the following two chains of inequalities:
 \begin{multline} \label{1stalternantInequality}
\lambda_2+\lambda_3+2\lambda_4
 <(\lambda_3+\lambda_4) +\min(\lambda_2+\lambda_3+1, \lambda_1+\lambda_4+1)\\
 \leq (\lambda_3+\lambda_4) +\max(\lambda_2+\lambda_3+1, \lambda_1+\lambda_4+1)
 < (\lambda_3+\lambda_4) +(\lambda_1+\lambda_3+2);
 \end{multline}
 \vskip-.2in
 
% \vskip-.3in
 \begin{multline} \label{2ndalternantInequality}
\lambda_2+\lambda_3+2\lambda_4
 <(\lambda_2+\lambda_4+1) +\min(\lambda_2+\lambda_3+1, \lambda_1+\lambda_4+1)\\
 \leq (\lambda_2+\lambda_4+1) +\max(\lambda_2+\lambda_3+1, \lambda_1+\lambda_4+1).
 \end{multline}
 A monomial $y^b x^a$ will make a nonzero contribution to 
 $g_{\mu,\nu,\lambda}$ if and only if 
$ b\leq \nu_2 \text{ and } b+a\leq \mu_2+\nu_2.$ 
In view of Eqn.~\eqref{yInequality},  the  condition on $\nu_2$ eliminates the possibility of any contribution to the Kronecker coefficient from the monomials in  the second  line of \eqref{explicitF22extraction}.
 Hence the subset of the remaining 4 monomials in Eqn.~\eqref{explicitF22extraction}  contributing to the Kronecker coefficient $g_{\mu, \nu, \lambda}$ is  determined by where the number $\mu_2+\nu_2$ falls in the consecutive  intervals determined by each of the inequalities Eqn.~\eqref{1stalternantInequality}.
 The  bounds on $\mu_2+\nu_2$ clearly  eliminate all but the first monomial, the atomic coefficient, in the first line of \eqref{explicitF22extraction}.

Finally consider the third case, $\lambda_2+\lambda_3+2\lambda_4 =
 \mu_2+\nu_2.$  
 Note that  $\nu_2\ge \lambda_2+\lambda_4+1$  is impossible
because it would force 
\[\mu_2=(\lambda_2+\lambda_4-\nu_2) +(\lambda_3+\lambda_4)\le
(\lambda_3+\lambda_4)-1 < \lambda_2+\lambda_4< \nu_2.\]
Hence $\nu_2\le \lambda_2+\lambda_4,$ and we are reduced to the first two cases.  This finishes the proof.
\end{proof} 
 
 From this we can easily deduce some conditions on the parts which ensure that the atomic Kronecker coefficients are an upper bound for the Kronecker coefficients.   In fact 
 Theorem~\ref{thm:AtomicIsMax} below states that NO restrictions on the parts of $\lambda$ are needed, as we will show in the next section.

\subsection{The atomic Kronecker coefficient is an upper bound for the Kronecker coefficient for the case $2-2-4$}

In this section we  show  that for   a triple of partitions $\lambda, \mu, \nu$ of the same integer, such that $\ell(\lambda)\leq 4, \ell(\mu), \ell(\nu)\leq 2,$ 
and $\mu_2\geq \nu_2 ,$ the atomic Kronecker coefficient $\tilde g_{\mu,\nu,\lambda}$ is always greater than or equal to the actual Kronecker coefficient $g_{\mu,\nu,\lambda}.$

We use our polyhedral geometry approach to prove this result.  Theorem~\ref{7termF22Alternant} and its applications showed how the Kronecker coefficient is completely determined by the functions $p_S(n,m).$  Our proof, depending heavily on the fact that the $p_S(n,m)$  are vector partition functions, consists of a careful analysis of the contributions of each term in the polynomial 
$P_\lambda$ in the proof of Theorem~\ref{7termF22Alternant}.

Our arguments will reveal  a remarkable relationship between the seven monomials in 
$P_\lambda.$    For brevity we will label the exponents of 
$y$ and $x$ appearing in Eqn.~\ref{explicitF22extraction} as follows:
\[b=\lambda_3+\lambda_4, a_0=\lambda_2+\lambda_4, 
a_1=\lambda_2+\lambda_3+1, a_2=\lambda_1+\lambda_4+1, 
a_3=\lambda_1+\lambda_3+2.\]
Combining Eqns.~\eqref{yInequality},~\eqref{xInequality},  we have the inequalities
\begin{equation}\label{expIneq}b\leq a_0<\{a_1, a_2\}<a_3.\end{equation}
The polynomial $P_\lambda$ is then
$P_\lambda=y^b(x^{a_0}-x^{a_1}-x^{a_2}+x^{a_3})
+y^{a_0+1}(-x^{b-1}+x^{a_1}+x^{a_2}).$

Recall that the first monomial, $y^b x^{a_0},$ is the one that determines the atomic Kronecker coefficient.
We will call this the atomic monomial.
The dependency digraph of Figure~\ref{fig:DepGraph} for the signed monomials 
in $P_\lambda$   is a consequence of Theorem~\ref{7termF22Alternant} and the inequalities~\eqref{expIneq},~\eqref{1stalternantInequality},~\eqref{2ndalternantInequality}. If  $M_1,$ $M_2$ are signed monomials,  a directed edge from node  $M_1$ to node $ M_2$ in the digraph signifies that if $M_1$ makes a nonzero contribution to the Kronecker coefficient (as described by Theorem~\ref{7termF22Alternant}), then so must the monomial $M_2$.
%
\begin{comment}
\begin{align}\label{DepGraph}
&y^{a_0+1} x^{a_1}\rightarrow -y^{a_0+1} x^{b-1} & 
  &y^{a_0+2} x^{a_2}\rightarrow -y^{a_0+1} x^{b-1} \\
 & &-y^{a_0+1} x^{b-1}\rightarrow y^b x^{a_0} &\\
& -y^b x^{a_1}\rightarrow y^b x^{a_0} & 
& -y^b x^{a_2}\rightarrow y^b x^{a_0}\\
&+y^bx^{a_3}\rightarrow  -y^b x^{a_1} & 
&+y^bx^{a_3}\rightarrow  -y^b x^{a_2}\\
&y^{a_0+1} x^{a_1}\rightarrow -y^{b} x^{a_1} &
  & y^{a_0+2} x^{a_2}\rightarrow -y^{b} x^{a_2}
\end{align}
\end{comment}
%

%%%%%%%%%%%%%%%%%%%%%%%%%
\begin{figure}
\begin{center}
\tikzstyle{line} = [draw, -latex']
\begin{tikzpicture}[scale=1.5]
  
  \node (M6) at (-1,6) {$\scriptstyle +y^{a_0+1} x^{a_1}$};
  \node (M7) at (1,6) {$\scriptstyle +y^{a_0+1} x^{a_2}$};
  \node (M5) at (0,5) {$\scriptstyle -y^{a_0+1} x^{b-1}$};
  \node (M1) at (0,4) {$\scriptstyle \mathbf{+y^b x^{a_0}}$};
  \node (M2) at (-1,3) {$\scriptstyle -y^b x^{a_1}$};
  \node (M3) at (1,3) {$\scriptstyle -y^b x^{a_2}$};
  \node (M4) at (0,2) {$\scriptstyle  +y^b x^{a_3}$};
  
   \path  [ultra thick,line, blue] (M4) -- (M3); \path [ultra thick,line] (M3) -- (M1);
  \path [ultra thick,line, red] (M4) -- (M2);  \path [ultra thick,line] (M2) -- (M1);
     \path [ultra thick,line]  (M5) -- (M1); 
     \path [ultra thick,line, red] (M6) -- (M5) ;
    \path [ultra thick,line, blue]  (M7) -- (M5);
     \path [ultra thick,line, blue] (M6) -- (M2); 
     \path [ultra thick,line, red]  (M7) -- (M3); 
 
\end{tikzpicture}
\end{center}
\caption{\small Dependency digraph for the monomials in $P_\lambda$  (the atomic monomial is in bold).  The blue and red arrows correspond to the two scenarios described in the proof of Theorem~\ref{thm:AtomicIsMax}.}
\label{fig:DepGraph}
\end{figure}

 We will examine the contribution to $g_{\mu,\nu,\lambda}$ of each of the three non-atomic monomials  in $P_\lambda$ with positive coefficient.  By Eqn.~(\ref{Exact}) from Theorem~\ref{thm:MAIN1}, this in turn will necessarily entail a detailed analysis of the vector partition function 
$p_S(n,m)$ of Section~\ref{F22}.    The final result exhibits the following surprising  phenomenon in the monomials of $P_\lambda.$  We will show that in fact, every non-atomic monomial with positive coefficient can be matched with a monomial with negative coefficient to yield a net nonpositive value (see the coloured arrows in Figure~\ref{fig:DepGraph}). For clarity of  exposition, the technical lemmas have been relegated to the Appendix at the end of the paper.

 \begin{theorem}~\label{thm:AtomicIsMax}  The atomic coefficient is an upper bound for the Kronecker coefficient in the case $2-2-4.$
 
  \end{theorem}

\begin{proof}
 The atomic Kronecker coefficient is determined by only  the first monomial
$y^b x^{a_0}.$  In order to prove that the result of the corresponding  coefficient extraction from $F_{2,2}$ is never less 
than the actual Kronecker coefficient, it suffices to show that the contribution of the  three remaining (non-atomic) positively signed  monomials, viz. 
$+y^{a_0+1} x^{a_1}, +y^{a_0+1} x^{a_2}, +y^bx^{a_3}$
 is offset by that of the three negative ones,
$ -y^{b} x^{a_1}, -y^{b} x^{a_2}, -y^{a_0+1} x^{b-1}.$

More precisely, we say that a positive monomial $+M_1$ is \textit{offset} by a negative monomial $-M_2$ if the contribution of 
$M_1-M_2$ to the Kronecker coefficient is \textit{nonpositive}. 

Lemmas~\ref{pSboundBinomCoeff} to~\ref{lastmonomial} in the Appendix will establish that one of the following two scenarios,  corresponding respectively to the blue arrows and the red arrows in Figure~\ref{fig:DepGraph}, \textit{must} occur.  

The contribution of 
\begin{enumerate}  
\item    $+y^{a_0+1} x^{a_1}$  is offset by $-y^{b} x^{a_1}$ AND
\item $+y^{a_0+1} x^{a_2}$ is offset by $-y^{a_0+1} x^{b-1}$ AND 
\item    $+y^b x^{a_3}$ is offset by $-y^b x^{a_2};$ 
\end{enumerate}
OR the contribution of 
\begin{enumerate}  
\item  $+y^{a_0+1} x^{a_1}$  is offset by $-y^{a_0+1} x^{b-1}$ AND 
\item    $+y^{a_0+1} x^{a_2}$ is offset by$-y^{b} x^{a_2}$  AND 
\item    $+y^b x^{a_3}$ is offset by $-y^b x^{a_1}.$ 
\end{enumerate}

The above two scenarios show that, apart from the monomial $y^bx^{a_0},$  whenever there is a contribution from a positively signed monomial in $P_\lambda$ to the Kronecker coefficient, there is an offsetting negatively signed monomial which also contributes, resulting in a net \textit{nonpositive} contribution.

This completes the proof that the monomial $y^bx^{a_0}$ gives the maximal contribution to the Kronecker coefficient, i.e. that $\tilde g_{\mu,\nu,\lambda}$ is an upper bound. \end{proof}

 \subsection{Bravyi's  vanishing conditions} 

 Given a partition $\lambda$, denote by  $\bar \lambda$  the partition obtained from $\lambda$ after deleting its first part. 
 Murnaghan  discovered  a necessary condition for the Kronecker coefficient $g_{\lambda, \mu, \nu}$ to be nonzero.
He showed that the following inequality has to hold:
\begin{align}\label{MurnaghanIn}
|\bar \lambda| \le |\bar \mu| + |\bar \nu|,
\end{align}
 Note that since the Kronecker coefficients are symmetric under permutations of the index, there are really three inequalities.
 
 The following stronger result (due to Bravyi) follows from our methods.  The reader may want to compare with Proposition \ref{ineqatomiccone22}.%\ref{Bravyi}.
\begin{proposition}[Bravyi~\cite{Kirillov:saturation, Bravyi}]
\label{nonzeroKron}  

Assume $\ell(\lambda)\leq 4, \ell(\mu), \ell(\nu)\leq 2.$  The Kronecker coefficient is zero if $\nu_2<\lambda_3+\lambda_4$ or $\mu_2+\nu_2<\lambda_2+\lambda_3+2\lambda_4 .$  Equivalently, if the Kronecker coefficient $g_{\mu, \nu,\lambda}$  is nonzero then Bravyi's inequalities~\eqref{Bravyi} are satisfied. \end{proposition}
\begin{proof} We assume without loss of generality that $\mu_2\geq \nu_2.$
We will use the polynomial $P_\lambda$ of Theorem~\ref{7termF22Alternant}.

First suppose $\nu_2<\lambda_3+\lambda_4$.  Then we have 
$\nu_2<\lambda_3+\lambda_4\le \lambda_2+\lambda_4.$  Examining the polynomial 
$P_\lambda$ in~\eqref{explicitF22extraction}, we see that none of the monomials $y^bx^a$ makes a contribution since the condition $b\leq \nu_2$ is violated for both exponents $b$ of $y$ in $P_\lambda.$

Now suppose $\mu_2+\nu_2<\lambda_2+\lambda_3+2\lambda_4.$  Observe that 
$\lambda_2+\lambda_3+2\lambda_4$ is precisely the sum of the exponents for the first monomials $y^bx^a$ in each of the first two lines of the polynomial in~\eqref{explicitF22extraction}. Hence the  condition $b+a\leq \mu_2+\nu_2$ is violated for these two monomials.   But the sum of exponents $b+a$ for each of the other monomials in~\eqref{explicitF22extraction} is strictly greater than the sum for the first monomial in each line, so the condition is violated for all the monomials in $P_\lambda.$ \end{proof}

\subsection{A closed formula for the reduced Kronecker coefficients}\label{section:reduced}

Murnaghan %\cite{ Murnaghan:1937,  Murnaghan:1938} 
  also observed that the sequences of Kronecker coefficients 
$ \big(g(\lambda+(k), \mu+(k), \nu+(k)) \big)_{k\ge0}$
 always  stabilize. 
   Their stable value is known as  the \emph{reduced Kronecker coefficient} and denoted by $
 \bar g_{\bar \lambda, \bar  \mu, \bar  \nu}
$, where, given a partition $\lambda$, we denote by $\bar \lambda$ the partition obtained from $\lambda$ deleting its first part.

 \begin{proposition}\label{Stability2parts}  Let $\lambda_2 \ge \mu_2\geq \nu_2.$ Assume $\lambda$ has at most two parts. Then $g_{\mu,\nu,\lambda}$ is independent of $\lambda_1$ as soon as $\lambda_1\geq \mu_2+\nu_2.$
Moreover, the stable value is $p_S(\nu_2, \nu_2+\mu_2-\lambda_2)-p_S(\nu_2, \nu_2+\mu_2-\lambda_2-1).$ Explicitly,  let $\ell = \nu_2+\mu_2-\lambda_2.$   Then:
  \begin{align}\label{reduced224}
  \bar g_{(\lambda_2),(\mu_2),(\nu_2)}&= \begin{cases}
\frac{\ell}{2} + \frac{ 3 + (-1)^{\ell}}{4} =  \left \lfloor \frac{\ell}{2} \right \rfloor +1 &\text{ if } \mu_2+\nu_2 \ge  \lambda_2\\
0 &\text{ if } \mu_2+\nu_2< \lambda_2
\end{cases}
 \end{align} 
 \begin{proof} We have $\lambda_3=\lambda_4=0.$
 For any statement $S$, we write $\delta(S) $ to mean 1 if $S$ is true and 0 otherwise.
 From Eqn.~(\ref{Exact}) of Theorem~\ref{7termF22Alternant}, the  Kronecker coefficient is
 \small{
 \begin{multline*}%\label{PropReduced}
 p_S(\nu_2, \nu_2+ \mu_2-\lambda_2)
 -p_S(\nu_2, \nu_2+ \mu_2-(\lambda_2+1))\\
 -p_S(\nu_2, \nu_2+ \mu_2-(\lambda_1+1))\cdot {\scriptstyle\mathbf{\delta(\mu_2+\nu_2\geq \lambda_1+1)}}
+  p_S(\nu_2, \nu_2+ \mu_2-(\lambda_1+2))\cdot {\scriptstyle\mathbf{\delta(\mu_2+\nu_2\geq \lambda_1+2)}}\\
-p_S(\nu_2-(\lambda_2+1), \nu_2-(\lambda_2+1)+\mu_2+1)
+p_S(\nu_2-(\lambda_2+1), \nu_2-(\lambda_2+1)+\mu_2-(\lambda_2+1))\\
+0\cdot p_S(\nu_2-(\lambda_2+1), \nu_2-(\lambda_2+1)+\mu_2-(\lambda_1+1))
 \end{multline*}}
 The zero coefficient in the last line is explained by the fact that $ (\lambda_1+\lambda_2)/2\geq \mu_2\geq \nu_2, $ and thus we always have $\mu_2+\nu_2\leq \lambda_1+\lambda_2.$  
 
The  hypothesis that  $\lambda_1\geq \mu_2+\nu_2$  eliminates the two terms with  $\lambda_1$ in their arguments,   establishing a stable value.  The reduced  Kronecker coefficient is thus given by 
 \begin{multline}\label{PropReduced}
 \mathbf{p_S}(\nu_2, \nu_2+ \mu_2-\lambda_2)
 -\mathbf{p_S}(\nu_2, \nu_2+ \mu_2-(\lambda_2+1))\\
-\mathbf{p_S}(\nu_2-(\lambda_2+1), \nu_2+\mu_2-\lambda_2)
+\mathbf{p_S}(\nu_2-(\lambda_2+1), \nu_2+\mu_2-2(\lambda_2+1))
 \end{multline}
 
Of these four terms, since $\lambda_2\ge\mu_2\ge\nu_2,$ the third and  fourth terms are immediately eliminated because the first argument is negative: 
 $\nu_2-\lambda_2-1\le -1.$   
 
 If $\mu_2+\nu_2< \lambda_2,$  both first and second terms are identically zero.
 
 If $\mu_2+\nu_2= \lambda_2,$ only the first term appears, but it must be 1 from the boundary values $p_S(n,0)=1=p_S(0,m)$  recorded in Section~\ref{SheilaExamples}.
 
If $\mu_2+\nu_2> \lambda_2,$  the first two vector partition functions both appear.  Since  $\lambda_2-\mu_2\ge 0$ implies $\nu_2\ge\ell=\nu_2+\mu_2-\lambda_2,$ 
 both  are computed using the quasipolynomial corresponding to Region I in Figure~\ref{fig:chambers}, and consequently  the reduced Kronecker coefficient is 
\[\mathbf{p_S}(\nu_2,\ell)-\mathbf{p_S}(\nu_2,\ell-1)
=\frac{\ell}{2} + \frac{ 3 + (-1)^{\ell}}{4}.\] 
The proof is now complete.
 \end{proof}
 \end{proposition}
 Using the symmetry of the Kronecker coefficients with respect to the three partitions, we immediately have:

\begin{figure}\center
\begin{minipage}[t]{.5\textwidth}
 \begin{tikzpicture}
\node[inner sep=0pt] (picture1) at (0,0) {\includegraphics[width=\textwidth]{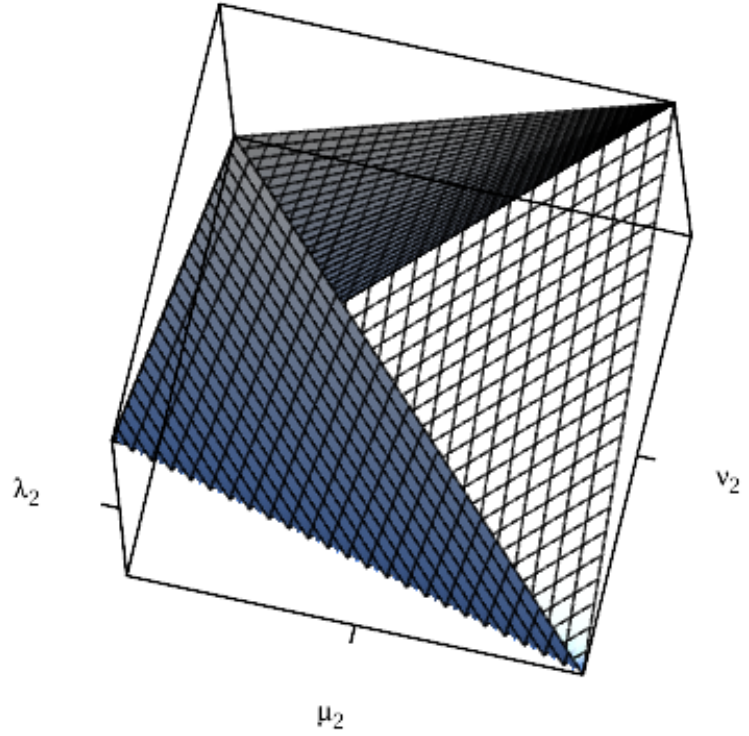}};%
\node[label=left:{$(0,0,0)$}] at (-1.8,-.5) {$\bullet$};

\end{tikzpicture}
\end{minipage}
\hfill
\begin{minipage}[b]{.38\textwidth}
\begin{tikzpicture}[scale=.4]
 \foreach \n in {1,2,3}{
        \node at ({\n*360/3+45}:4cm) (n\n) {};
        \draw[color=white] (0,0)--(n\n);
        
    }
    \fill[opacity=.4] (0,0) -- (45:4cm) -- (165:4cm);
  \fill[opacity=.2, color=- gray] (0,0) -- (45:4cm) -- (285:4cm);
     \fill[opacity=.6, color=blue!20!gray] (0,0) -- (285:4cm) -- (165:4cm);
  \draw[color=white] (n1) -- (n2) -- (n3) -- (n1);
  \node at (105:1cm) {$III$};
  \node at (-15:1cm) {$II$}; 
  \node at (225:1cm) {$I$};
\end{tikzpicture}
\vspace{2cm}
\mbox{}
\end{minipage}

\caption{The chamber complex for the reduced Kronecker coefficients indexed by three one--row shapes. The reduced coefficients are zero outside  the tetrahedra,
one on all the boundaries, and grow linearly as we move towards the center of one of the chambers. The interior of the cone is divided into three chambers.}
\label{fig:RKCchamber}
\end{figure}

 \begin{corollary} Fix $\lambda_2, \mu_2, $ and $\nu_2$. Let $a=\max(\lambda_2, \mu_2, \nu_2)\geq b\geq c=\min(\lambda_2, \mu_2, \nu_2)$ be a total ordering of $\lambda_2, \mu_2, \nu_2.$ 
 Set  $\ell =b+c-a.$  Then
   \begin{align}
   \bar g_{(\lambda_2),(\mu_2),(\nu_2)}&=    \bar g_{(a),(b),(c)}
   =
   \begin{cases}
\frac{\ell}{2} + \frac{ 3 + (-1)^{\ell}}{4} =  \left \lfloor \frac{\ell}{2} \right \rfloor +1 &\text{ if } \ell \ge  0\\
0 &\text{ if } \ell < 0.
\end{cases}
 \end{align} 
The chamber complex for this quasipolynomial is illustrated  in Figure~\ref{fig:RKCchamber}. 
% Used to be Figure~\ref{fig:chambers}.   Corrected 2021-4-17
The walls are the hyperplanes $I) : \mu_2+\nu_2=\lambda_2$,
$II ) : \mu_2+\lambda_2=\nu_2$, $III ) : \lambda_2+\nu_2=\mu_2.$ The  reduced Kronecker coefficient  indexed by points on any of these walls always has value  equal to one.
 \end{corollary}
 
We have obtained the counting function for the number of integer points in the  one-dimensional polytope of Figure~\ref{fig:zeroexample}, which was  studied in 
Example~\ref{ZeroExample}.

\subsection{The relative positions of the cones associated to the Kronecker, the reduced Kronecker and the Littlewood--Richardson coefficients}\label{RKC}

Identify a triple of partitions of lengths  $\le a, b, c$ (respectively)  with a point in $\mathbb{Q}^{a+b+c}$. 
 The set of triples of partitions whose
corresponding Kronecker coefficient is nonzero is known to have the structure of a finitely
generated semigroup (\cite{Christandl-PhD, Klyachko, Manivel:asymptotics1}). This semigroup
generates a rational polyhedral cone, called the  \emph{Kronecker cone}  and denoted by $PKron_{a,b,c}$.  
Its walls (i.e. facets) are described by a finite set of inequalities.
Can we find these inequalities? Is this cone saturated, or do there exist \emph{holes}, that is,
 points where the Kronecker coefficient is zero,  inside it? If so, where are the holes located? What is the relation between the Kronecker coefficients and other 
 important families of coefficients such as the Littlewood--Richardson coefficients or the reduced Kronecker coefficients.
 
We will use an unexpected discovery of Murnaghan to explore these issues: In the particular case where $\bar \lambda=\bar\mu+\bar\nu$ (that is, when Eq~\eqref{MurnaghanIn} is an equality), and when the first parts of the partitions are ``big enough", the Kronecker coefficient $g_{\lambda, \mu, \nu}$ coincides with the Littlewood-Richardson
coefficient $c^{\bar \lambda}_{\bar \mu, \bar \nu}$.  
Equivalently, Murnaghan's result can be expressed in terms of the reduced Kronecker coefficients,
$
 \bar g_{\bar \lambda, \bar  \mu, \bar  \nu}= c_{\bar \lambda, \bar  \mu, \bar  \nu}
$.

Given a family of coefficients indexed by triples of partitions, we define a cone (with the same name) as the polyhedral cone generated by its nonzero values. 
We ask for the relation between the positions of these different cones  (the Kronecker, the reduced Kronecker, and the Littlewood--Richardson coefficients) inside of the atomic cone.
 
The atomic Kronecker coefficients attain their minimum nonzero value (one) at its boundary:
 $
\mu_2+\nu_2-\lambda_2-\lambda_3-2\lambda_4=0
$ or $\nu_2-\lambda_3-\lambda_4=0$ (compare with inequality~\eqref{Bravyi}).   %(See also the explicit formulas for the coefficients of ${F}_{2,2}$ in  Section~\ref{F22}.)  
Being the solution of a vector partition function problem, it follows that  when we dilate 
  the  three indexing partitions,  the values of the atomic Kronecker coefficients inside the cone are always  increasing. 

\begin{remark}\label{lambda1islarge}  On the face defined by 
$
\mu_2+\nu_2-\lambda_2-\lambda_3-2\lambda_4=0
$, the atomic Kronecker coefficient coincides with the Kronecker coefficient; see  Corollary~\ref{atomicisKron}.   Furthermore, they are always equal to one.
A triple of partitions $(\lambda, \mu, \nu)$ of the same weight is stable if $g_{(k\lambda, k\mu, k\nu)}$ equals 1 for all $k.$   
We have shown that all Kronecker coefficients corresponding to triples in this face are stable triples. 
 Stable triples are relevant because the sequences $g_{\alpha+n\lambda, \beta+n\mu, \gamma+\nu}$ stabilize for $n$ sufficiently large, see \cite{Stembridge}.
\end{remark}

 \begin{theorem}\label{LRfaceofKron}
 The Littlewood--Richardson cone coincides with  the intersection of the hyperplane $\lambda_4=0$ and the face   of the Kronecker cone defined by the first Bravyi inequality, 
$ 
  \lambda_2+\lambda_3+2\lambda_4 = \mu_2 + \nu_2  
 $.
% and the hyperplane $\lambda_4=0$.
 \end{theorem}

\begin{proof}
Let $c_{\alpha, \beta}^\gamma=\bar g_{\alpha, \beta}^\gamma=g_{\lambda,\mu,\nu}$ with $\alpha=\bar \mu$,  $\beta=\bar \nu$  and $\gamma=\bar \lambda$. 
 We want to
see where $(\lambda,\mu,\nu)$ sits in relation with the Kronecker cone.

Suppose that  $c_{\alpha, \beta}^\gamma>0$.
Since $\alpha$ and $\beta$ have just one part, 
Pieri's rule tells us that the length of $\gamma$ can be at most two, and hence 
$\lambda_4=0$, and $ \lambda_2+\lambda_3 = \mu_2 + \nu_2  $.
%Setting $\lambda_4=0$ into  Murnaghan's condition~\eqref{MurnaghanIn} we conclude that 
%we are on the wall of the atomic cone described in the intersection described in the theorem. 
On the other hand, using the inequalities described in Corollary~\ref{atomicisKron} we easily see that all atomic Kronecker coefficients in
this wall are indeed Kronecker coefficients.

Let $N$ be any number such that the triple of sequences $(N-\lambda_2,\lambda_2,\lambda_3),(N-\mu_2,\mu_2),(N-\nu_2,\nu_2)$ are partitions.
It remains to see whether  
$g_{(N-\lambda_2,\lambda_2,\lambda_3),(N-\mu_2,\mu_2),(N-\nu_2,\nu_2)}$ is reduced. 
For this, we will use the bound described in Theorem 1.5 of~\cite{BOR-JA}. It says that such a  Kronecker coefficient is stable as soon
as $N\ge\lambda_2+\mu_2+\nu_2+\left\lfloor \frac{\lambda_3}{2} \right\rfloor=2\lambda_2+\left\lfloor \frac{\lambda_3}{2} \right\rfloor$. But since $(N-\lambda_2,\lambda_2,\lambda_3)$
has to be a partition, the smallest possible value for $N$ will be that one that corresponds to partition $( \lambda_2,\lambda_2,\lambda_3)$. That is,  $2\lambda_2+\lambda_3 \ge 2\lambda_2+\left\lfloor \frac{\lambda_3}{2} \right\rfloor.$ 
   \end{proof}

We  ask for the position of those nonzero Littlewood--Richardson coefficients (inside the Kronecker cone) coming from the  identities:
$
 g_{\mu,\nu,\lambda} = \bar g_{\bar \mu,\bar \nu, \bar \lambda} =  c_{ \bar \mu,\bar  \nu}^{\bar  \lambda}.
$
 Now the Littlewood-Richardson coefficient $c^{(\lambda_2, \lambda_3)}_{(\mu_2), (\nu_2)}$ is nonzero iff the skew-shapes $(\lambda_2, \lambda_3)/(\mu_2)$ and $(\lambda_2, \lambda_3)/ (\nu_2)$ are horizontal strips, or equivalently iff 
 $ \lambda_2\geq \mu_2\geq \lambda_3 \text{ and } 
  \lambda_2\geq \nu_2\geq \lambda_3,$ 
 which in turn is equivalent to saying $\mu_2$ and $ \nu_2$ lie in the interval $[\lambda_3, \lambda_2].$ Hence, when the LR coefficient is nonzero,   we must have (since $\lambda_1\geq \lambda_2\geq \lambda_3\geq \lambda_4$),
 $$|\mu_2-\nu_2|\leq \lambda_2-\lambda_3\leq \min(\lambda_1-\lambda_3,  \lambda_2-\lambda_4).$$  We have recovered another result of Bravyi, his 
 third,  and last, inequality.

\begin{remark}
We have the following implications.  If a Kronecker coefficient is atomic then it is  reduced, the reason being that   the value of $\tilde g$ does not depend on the first part of $\lambda$.
On the other hand, if a Kronecker coefficient satisfies the equality of Murnaghan's condition, then it is reduced.
This is Theorem~\ref{LRfaceofKron}.\end{remark}

\section{The vector partition function $F_{n,m}$}
\label{sec-Fnm}
\label{nmKronecker}
Having completed our analysis of the $n=m=2$ case, we examine the extent to which these results generalize. First we show that we can make a variable substitution to convert
$\frac{{a_{\delta_n}[X]}{a_{\delta_m}[Y]}}{a_{\delta_{nm}}[XY]}$ into a vector partition function. This is the result of Theorem~\ref{thm:MAIN2}.

Let $S(p)$ denote the smallest monomial in $p$ with respect to the lexicographic order. Repeating the reasoning given for the case $n=m=2$ in Section 3, 
we obtain for the general situation 
the following identity:
\small{
\begin{align}\label{tildeKronecker}
\frac{{a_{\delta_n}[X]}{a_{\delta_m}[Y]}}{a_{\delta_{nm}}[XY]} S( a_{\lambda+\delta_{nm}}[XY])
&=  \sum_{\mu, \nu} \tilde g_{\mu,\nu,\lambda} \,  S(a_{\mu+\delta_n}[X] ) S( a_{\nu+\delta_m}[Y]) +\hot
\end{align}}
We proceed by first expanding all the Vandermonde determinants involved as a product of linear binomial factors. We want to factor the binomial terms so that we obtain a product of  terms of the form $(1-M)$ where $M$ is a Laurent monomial. We  can do this in such a way that the resulting Laurent series converges in a nonzero domain if we follow the lexicographic ordering, and always factor the smallest monomial in each binomial.  The argument is similar to the one for $F_{2,2},$ see the discussion following ~\eqref{F22series}.

We consider the special alphabets $X=\{1,x_1,x_2,\ldots ,x_{n-1}\},\quad$ $Y=\{1,y_1,y_2,\ldots, y_{m-1}\}, n,m\geq 2.$ Then $XY=\{1, x_i, y_j, x_iy_j:
1\leq i\leq n-1, 1\leq j\leq m-1\}.$
The set $XY$ is ordered as follows:
\begin{equation}\label{ordering}
1>x_i>x_{i+1}, y_j>y_{j+1}, x_i>y_j, x_i y_j>x_ky_\ell \text{ if }i< k,\text{ or } i=k \text{ and } j<l.
\end{equation}
The following claim is clear.
\begin{lemma}The smallest term  in $S( a_{\lambda+\delta_{nm}}[XY])$, with respect to the lexicographic ordering, is the product of the monomials in the main diagonal of the  matrix of the alternant $a_{\lambda+\delta_{nm}}[XY].$
\end{lemma}

 Similarly, to compute the smallest term, with respect to the lex ordering, in each of the two remaining alternants $S( a_{\mu+\delta_{n}}[X])$ and $S( a_{\nu+\delta_{m}}[Y])$, we take the product of the monomials in the main diagonal of the corresponding  matrices.
We  obtain a Laurent series
\begin{align*}\label{Laurentseries}
 \sum_{\mu, \nu} g_{\mu,\nu,\lambda} \,  {\bf x}^{\,\,\,l_1(\mu,\nu,\lambda)} {\bf y}^{\,\,\,l_2(\mu,\nu,\lambda)} 
\end{align*}
where $l_1(\mu,\nu,\lambda)$ and $l_2(\mu,\nu,\lambda)$ are linear combinations of the parts of $\mu, \nu$ and $\lambda$. It is a product of binomial terms of the form 1 minus a Laurent monomial. Finally, we perform a change of basis which we describe in detail below, to ensure that we get a convergent Taylor series expansion.

For example, for $n=m=3$, the substitution is
$x_{1} = s_{1} t_{1}, $ 
$ x_{2} = s_{1} s_{2} t_{1}^{2}, $
$ y_{1} = s_{0} s_{1} s_{2} t_{1}^{2}, $ and 
$ y_{2} = s_{0} s_{1} s_{2} t_{1}^{3}$
More precisely, in order to guarantee convergence of our series, we assume in Eqn.~\eqref{ordering} that 
$1>|x_1|>|x_2|>..>|x_{n-1}|>|y_1|>|y_2|>...>|y_{m-1}|
 >|x_1y_1|>|x_1y_2|>...>|x_1y_{m-1}|\\
\ldots 
  >|x_{n-1}y_1|>|x_{n-1}y_2|>...>|x_{n-1}y_{m-1}|
$.

We define the rational function $G_{n,m}$  by
$G_{n,m}=\frac{{a_{\delta_n}[X]}{a_{\delta_m}[Y]}}{a_{\delta_{nm}}[XY]}.$
Observe that we have the Vandermonde expansion
$$a_{\delta_n}[X]=\prod_{i=1}^{n-1}(1-x_i)\prod_{1\le i <j\le n-1} (x_i-x_j),$$
and similarly for the second alternant $a_{\delta_m}[Y].$
For the alternant in the denominator we have 
%\begin{multline*}
$a_{\delta_{nm}}[XY]%=\prod_{i=1}^{n-1}(1-x_i)
%\prod_{1\le i <j\le n-1} (x_i-x_j)\cdot \prod_{i=1}^{m-1}(1-y_i)\prod_{1\le i <j\le m-1} (y_i-y_j)\\ \cdot A \cdot B\cdot C\cdot D\cdot E\cdot F\\
=a_{\delta_n}[X]\cdot a_{\delta_m}[Y]\cdot A \cdot B\cdot C\cdot D\cdot E\cdot F,$
%\end{multline*}
%
where 
$$A=\prod_{j=1}^{m-1}\prod_{i=1}^{n-1}(x_i-y_j),  \,\,
B=\prod_{i=1}^{n-1}\prod_{j=1}^{m-1}(1-x_iy_j)$$
and
$C= \prod_{i=1}^{n-1}\prod_{j=1}^{m-1}(x_i-x_i y_j)\cdot 
\prod_{j=1}^{m-1}\prod_{i=1}^{n-1}(y_j-x_i y_j)$
\[=\left(\prod_{j=1}^{m-1}(1-y_j)\right)^{n-1}(\prod_{i=1}^{n-1} x_i^{m-1})  % total of 2(m-1)(n-1) factors
\cdot \left(\prod_{i=1}^{n-1}(1-x_i )\right)^{m-1} (\prod_{j=1}^{m-1}y_j^{n-1}), % total of 2(m-1)(n-1) factors
\]
\[D=\prod_{\stackrel {k=1}{k\neq i}}^{n-1}\prod_{i=1}^{n-1}\prod_{j=1}^{m-1} (x_k-x_iy_j)
\cdot \prod_{\stackrel {k=1}{k\neq j}}^{m-1}\prod_{j=1}^{m-1}\prod_{i=1}^{n-1} (y_k-x_iy_j),\qquad
E=\prod_{j\neq \ell=1}^{m-1}\prod_{1\le i<k\le n-1} (x_i y_j -x_k y_\ell),\]
\begin{multline*}
F= \prod_{j=1}^{m-1}\prod_{1\le i<k\le n-1} (x_i y_j -x_k y_j)
\cdot \prod_{i=1}^{n-1}\prod_{1\le j<\ell\le m-1} (x_i y_j -x_i y_\ell)\\
=\prod_{j=1}^{m-1}y_j^{n-1\choose 2}\prod_{1\le i<k\le n-1} (x_i  -x_k )^{m-1} % total of 2(m-1)(n-1  choose 2) factors
\cdot \prod_{i=1}^{n-1}x_i^{m-1\choose 2}\prod_{1\le j<\ell\le m-1} ( y_j -y_\ell)^{n-1}% total of 2(m-1)(n-1  choose 2) factors
.\end{multline*}

It follows that the quotient of alternants $G_{n,m}$ simplifies to $\dfrac{1}{ABCDEF}.$

\smallskip

Note that each factor in $A, C, D, E, F$ can be rewritten in the form $(1-M)$ where $M$ is a \textit{Laurent} monomial in the $x_i$ and the $y_j.$ The factors of $B$ are already in this form.
For instance, in $E$ we can rewrite each factor as 
$x_iy_j-x_ky_\ell=x_iy_j(1-x_ky_\ell x_i^{-1} y_j^{-1}).$

Thus, the following definition for $\bar{F}_{n,m}(X,Y)$  makes sense.
\begin{definition}[$\bar{F}_{n,m}(X,Y)$ ]\label{defineFnm}

There are positive integers $a_i, b_j$ such that in 
the product 
 \[G_{n,m}  \prod_{i=1}^{n-1}x_i^{a_i} \prod_{i=j}^{m-1}y_j^{b_j}, \]
all  factors are of the form $(1-M)^{-1}$ where $M$ is a Laurent monomial in the $x_i$ and the $y_j.$ We define $\bar{F}_{n,m}(X,Y)$ to be this product, i.e. we have 
\begin{equation}\label{Laurentseries}
G_{n,m}\prod_{i=1}^{n-1}x_i^{a_i} \prod_{i=j}^{m-1}y_j^{b_j}
=\bar{F}_{n,m}.
\end{equation}
\end{definition}

 We now show that there is a different set of $(n+m-2)$ variables $s_i, i=0, \ldots n-1, t_j, j=1,\ldots, m-2,$ such that by effecting a judicious (and non-obvious)  change of variables, $\bar{F}_{n,m}(X,Y)$ becomes a product of factors of the form $(1-M)^{-1}$ where each Laurent monomial $M$ in $X,Y$ is a  monomial with nonnegative exponents in the new variables $S,T.$  In other words, $F_{n,m}(S,T)$ is a vector partition function in the new variables.
 
We claim that the quotients of consecutive terms in the 
sequence \eqref{ordering} become monomials (and not Laurent monomials), after setting, for each  $1 \le i \le n-1$ and $1 \le j \le m-1$,
 \begin{align} \label{changeofvars}
 \begin{cases}
&x_{i}=s_1 s_2.. s_{i} (t_1 t_2 ... t_{m-2})^{i}\\
&y_{j}= (s_0 s_1 ... s_{n-1})  (t_1 t_2 ... t_{m-2})^{n-1} t_1 t_2...t_{j-1} \\
\end{cases}.
         \end{align}
 Note when $n=m=2$, there are no  $t_i$ variables, and we recover the substitution for $F_{2,2}$ in Section~\ref{F22}.
 We have 
\begin{multline}\label{basisvectorsFnm}
\frac{x_{i+1}}{ x_i }= s_{i+1} t_1 t_2 ... t_{m-2}, \quad
 \frac{y_{j+1}}{ y_j} = t_j, \ 1\le j  \le m-2,\\  
 \frac{y_1}{x_{n-1}} = s_0,  \quad \frac{x_1 y_1  }{ y_{m-1}} = s_1,\ 
  \frac{x_i y_1  }{  x_{i-1}y_{m-1}} = s_i ,\   2\le i\le n-1.
 \end{multline}
This establishes our claim.  Hence we have proved: 
\begin{theorem} 
\label{changeofvariables}\label{thm:AlwaysVecPart}\label{thm:MAIN2}
Let $F_{n,m}(S,T)$ be the series obtained after performing the previous substitutions in the series~\eqref{Laurentseries}; its domain of convergence is $\{|s_i|<1, |t_j|<1: 0\le i\le n-1, 1\le j\le m-2\}.$  Then $F_{n,m}$ is a vector partition function.
\end{theorem}
From the preceding discussion we can also conclude:
\begin{corollary}\label{matrixFnm}  Let $A_{n,m}$ be the matrix associated to the vector partition function $F_{n,m}$, as
in Section~\ref{sec-Polytopes}. Then  %in Section~\ref{VectorPartition}.   Then 
\begin{enumerate}
\item the largest entry is $2n-1;$
\item the number of columns  is $\binom{nm}{2}-\binom{n}{2}-\binom{m}{2}$;
\item the number of rows is $m+n-2;$
\item all the basis vectors appear in the columns of $A_{n,m};$
\item the rank of the matrix is $m+n-2.$
\end{enumerate}
\end{corollary}

\begin{proof} 
The largest entry is obtained by examining the largest possible exponent  of the variables $s_i$ or $t_j$ in the monomials $M$ occurring in the factors $(1-M)$ of $F_{n,m}.$ We have, from the product $B$ of the preceding proof, for $ i\le n-2, j\le m-1,$ the monomial 
\small{
\[ x_i y_j=(s_1\ldots s_i)(t_1t_2\ldots t_{m-2})^i\cdot (s_0s_1\ldots s_{n-1})
(t_1t_2\ldots t_{m-2})^{n-1}(t_1t_2\ldots t_{j-1})\]}
and clearly the largest exponent here occurs for each of $t_1,\ldots , t_{j-1},$ and it equals $i+(n-1)+1=n+i\le 2n-1.$  The maximum exponent $2n-1$ occurs in the monomial $x_{n-1} y_j.$
Examining the products other than $B,$ we see that all other monomials involve dividing by $x_i$ or $y_j$ or both, so it is clear that they cannot yield a larger exponent.

The number of columns in the matrix $A_{n,m}$ equals the number  of linear factors in $a_{\delta_{nm}}[XY]$ minus the number of linear factors in $a_{\delta_n}[X]$ minus the number of linear 
factors in $a_{\delta_m}[Y];$ since these are all Vandermonde determinants, the second result follows.  For the third  result, observe that the number of rows is simply the number of variables in the set $\{s_i, 0\le i\le n-1, t_j, 1\le j\le m-2\}.$ 

For the last two statements, observe that Eqn.\eqref{basisvectorsFnm} in the preceding proof establishes that all the basis vectors appear as columns of the matrix  $A_{n,m}$, since all the variables $s_i$ and $t_j$ occur as quotients when converting the factors of the Vandermonde in the products $A$-$F$ into the form $(1-M)$ in Eq.~\eqref{basisvectorsFnm}. Hence the rank is the number of rows of the matrix.  \end{proof}

\subsection{The degree of the Kronecker quasipolynomial}
\label{sec:DegreeandPeriod}

We have shown that the Kronecker function $\kappa_{n,m,nm}$ is a quasipolynomial on affine domains. However, a deep theorem of Meinrenken and Sjamaar \cite{qr0} says that $\kappa_{n,m,nm}$ is in fact a piecewise quasipolynomial. This result seems to be unattainable by our methods.
However, we immediately obtain information about the 
degree  of the Kronecker quasipolynomial $\kappa_{n,m,nm}$. Let $X$ be an alphabet of size $n$, and $Y$ an alphabet of size $m$, and let 
\begin{align*}
d&=\frac{n^2m^2}{2}-\frac{n^2}{2}-\frac{m^2}{2}-\frac{nm}{2}-\frac{n}{2}-\frac{m}{2}+2
\end{align*}

\begin{theorem}\label{thm:DegreeandPeriod}
The degree of the piecewise quasipolynomial Kronecker function $\kappa_{n,m,nm}$ is always $\le d$. 
\end{theorem}
\begin{proof} The degree of  $\kappa_{n,m,nm}$ is bounded by 
the dimension of the null space of $A_{n,m}.$ 
It is thus equal to the number of columns minus the rank.  By Corollary~\ref{matrixFnm}, this is just $d.$ 

Since Kronecker coefficients are linear combinations of different shifts of this vector partition function, these bounds apply in general.
\end{proof}

The degree of  $\kappa_{n,m,nm}$ has  been obtained by  Baldoni, Vergne, and Walter~\cite{BaldoniVergneWalter, VergneWalter} using the language of moment maps.

In addition to being completely elementary, another advantage of our approach is that the dimension of the polyhedral cones and their ambient spaces
 involved in the calculation are the minimal possible ones, as they
coincide with the degree of the quasipolynomial.

\begin{example}
The domain of convergence of the vector partition function $F_{2,3}$ is $|x_1y_2|<|x_1y_1|<|y_2|<|y_1|<|x_1|<1.$  $F_{2,3}$ counts  nonnegative integer solutions to $A_{2,3}{\bf x}={\bf n}$, with $A_{2,3}$ equal to 
\[
%\mkern-100mu 
\left(\begin{array}{rrrrrrrrrrrr}
1 & 0 & 0 & 1 & 0 & 0 & 1 & 0 & 1 & 1 & 1  \\
0 & 1 & 0 & 0 & 1 & 1 & 1 & 1 & 1 & 2 & 2  \\
0 & 0 & 1 & 1 & 1 & 1 & 1 & 2 & 2 & 2 & 3 
\end{array}\right)
\]
\end{example}
The dimension of the solution space is rather large. The polytopes involved have dimension $8,$ making them very hard to visualize.
However, some interesting phenomena can be observed by  looking at the restriction of this system of equations to the positive orthant.
 Recall that we are looking for nonnegative solutions to $A_{2,3}{\bf x}={\bf n}$. Let ${\bf n}=(n_1, n_2, n_3)$.

If $n_3=0$ , since we are only considering nonnegative linear combinations of the columns of the matrix, none of the columns other than the first two can appear. We obtain the restricted matrix \small{$A_3=\left(\begin{array}{rr} 1 & 0 \\ 0 & 1 \end{array}\right)$}, and $p_{A_3}(n_1,n_2)=1$ is a constant polynomial.  Here we use the notation of  Section~\ref{sec-Polytopes}  
%Theorem~\ref{Blakley}
 for the quasipolynomial $p_A(\bb) $ associated to the polytope defined by the  solution space of the matrix equation $A{\bf x}={\bb}.$ 

On the other hand, if $n_2=0$, we can discard any column where the second entry is not zero. In this case the restricted matrix 
is $A_2=\left(\begin{array}{rrr} 1 & 0 &1 \\ 0 & 1 &1 \end{array}\right)$, and $p_{A_2}(n_1,n_3)$ is a linear  polynomial: We need to solve  the system of inequalities $x_3 \le n_1, x_3 \le n_3$.  Hence $p_{A_2}(n_1,n_3)=1+\min(n_1, n_3).$
%The solution will depend on which one, $n_1$ or $n_3$  is the smallest, and it will be that number plus one.

Finally, if  $n_1=0$, the restricted matrix is   \small{$A_1=\left(\begin{array}{rrrrr}1 & 0 & 1 &1 &1 \\ 0& 1 & 1 & 1 & 2 \end{array}\right)$,} and $p_{A_1}(n_2,n_3)$ is a  cubic  quasipolynomial.

Note that  the atomic Kronecker coefficients are identically one {\em only} on the facet defined by  $n_3=0 $. Contrast this with the situation for $F_{2,2}$, where the coefficients are identically one on {\em both} facets: from Figure~\ref{fig:chambers}, we see that 
  $p_{A_{2,2}}(n,m)=1$ if $n=0$ or $m=0.$  %walls=facets

\section{Appendix: Lemmas for the proof of Theorem~\ref{thm:AtomicIsMax}}

We now establish the technical lemmas needed on the monotonic behaviour of the  function $p_S(n,m)$, in order to prove 
Theorem~\ref{thm:AtomicIsMax}.    For brevity, 
throughout these arguments, we will write $c(m)$ for the expression $\frac{7}{8}+\frac{(-1)^m}{8}.$  Note that $c(m)\leq 1$ for all $m.$  
 
 \begin{lemma}\label{pSboundBinomCoeff} The partition function $p_S(n,m)$ satisfies 
\[p_S(n,m)\le p_S(n,m')=\binom{n+2}{2} \text{ whenever }m'\geq 2n.\]
\end{lemma}
\begin{proof} We have three cases.

\noindent
\textbf{Case 1:} Suppose $n\in[0, \frac{m}{2}].$ Then we claim that $p_S(n,m)=\binom{n+2}{2}=p_S(n, m')$ for all $m'\ge 2n.$ 
This is just a consequence of the definition.

\noindent
\textbf{Case 2:}
 Suppose $n\in(\frac{m}{2},m).$  We must show that $p_S(n,m)\le
p_S(n,m')$ for all $m'\ge 2n.$ 

  From Figure~\ref{fig:chambers}, when $\frac{m}{2}\leq n<m,$ $p_S(n,m)$ is given by the formula for Region III, while $p_S(n,m')$ is given by the binomial coefficient $\binom{n+2}{2}.$  Inspecting the third figure in Figure~\ref{fig:polytopes}, and using the fact that the $p_S(n,m)$ count lattice points in the appropriate regions, it is immediate that the difference $p_S(n,m')-p_S(n,m)$ is nonnegative for $n$ in this interval and $m'\geq 2n.$
  
\begin{comment}
We compute the difference $p_S(n,m')-p_S(n,m)$ as a function $f(n)$ of $n$ for $n$ in the interval $(\frac{m}{2},m).$ 
This is 
\begin{align*} &f(n)=\frac{n^2}{2}+\frac{3n}{2}+1-nm+\frac{n^2}{2}+\frac{m^2}{4}-\frac{n+m}{2}-\frac{7}{8}-\frac{(-1)^m}{8}\\
&=n^2+n+1-c(m)+\frac{m^2}{4}-\frac{m}{2}-\frac{mn}{4},\quad \text{where } 1-c(m)=\frac{1-(-1)^m}{8}=0\text{ or } \frac{1}{4}.
\end{align*}
Hence the derivative $f'(n)=2n+1-m>0$ in the domain of $n,$ 
so $f(n)$ is an increasing function of $n$ in this interval.
But $f(\frac{m}{2})=1-c(m)$ which is nonnegative, so $f(n)$ is nonnegative for all $n$ in the interval $(\frac{m}{2},m).$ 
\end{comment}

\noindent
\textbf{Case 3:}
Suppose $n\ge m\ge 0.$  We must show that $p_S(n,m)\le \binom{n+2}{2}=p_S(n,m')$ for all $m'\ge 2n.$ Again this is immediate by the same geometric argument, inspecting the first and third figures in Figure~\ref{fig:polytopes}.

\end{proof}

\begin{comment}
  We compute
\begin{equation*} p_S(n,m) =\frac{m^2}{4} +m +c(m)\leq \frac{n^2}{4}+n +1<\frac{n^2}{2}+\frac{3n}{2} +1=\frac{(n+2)(n+1)}{2}.
%
\end{equation*}
\end{comment}

\begin{lemma}\label{pSRegionIII2nd} Suppose  $\frac{m}{2}<\frac{M}{2}<n<m<M.$ Then $p_S(n,M)-p_S(n,m)\geq 0.$
\end{lemma}
\begin{proof} Both partition functions are computed according to the formula for Region III in Figure~\ref{fig:chambers}, and hence count lattice points in a convex polytope (see Figure~\ref{fig:polytopes}).  They are therefore increasing functions in each argument.
\end{proof}
\begin{comment}
  We compute:
\begin{align*} &p_S(n,M)-p_S(n,m)=n(M-m)-\frac{1}{4}(M^2-m^2) +\frac{1}{2}(M-m)+c(M)-c(m)\\
&=(M-m)(n-\frac{M+m}{4}+\frac{1}{2}) +c(M)-c(m),
\end{align*}
and this is easily seen to be nonnegative since $c(M)-c(n)=0, \pm 1/4,$ and $n-\frac{M+m}{4}=\frac{(2n-M)+(2n-m)}{4}>0.$ 
\end{comment}

\begin{lemma}\label{pSRegionIII}  Fix $k\geq 0.$ Then  $p_S(n, n+k),$ 
for $0\leq k\leq n,$ is an increasing function of $n.$ 
\end{lemma}
\begin{proof}  From Section~\ref{F22}, we see that the conditions on    $k,n,$ imply that $p_S(n, n+k)$ corresponds to Region III in Figure~\ref{fig:chambers}. %$n+k\leq 2n.$
As before, since the function $p_S(n, n+k)$ counts lattice points in a convex polytope (see the third figure in Figure~\ref{fig:polytopes}), it is an increasing function of $n.$ \end{proof}

\begin{comment}
Hence we have 
\begin{multline}\label{pSRegionIIIformula}p_S(n, n+k)
=n(n+k) -\frac{n^2}{2}-\frac{(n+k)^2}{4}+\frac{2n+k}{2} +c(n+k)\\
=\frac{n^2}{4}+n(\frac{k}{2}+1) -\frac{k^2}{4}+\frac{k}{2}+c(n+k),\end{multline}
where we have written $c(m)=\frac{7+(-1)^{m}}{8}.$
Let $n_2>n_1.$ Then 
\begin{align*}&p_S(n_2, n_2+k)-p_S(n_1, n_1+k)\\
&=\frac{n_2^2-n_1^2}{4}+(n_2-n_1)(\frac{k}{2}+1)+c(n_2+k)-c(n_1+k)\\
&=(n_2-n_1)(\frac{n_2+n_1}{4}+\frac{k}{2}+1) +(-1)^k\frac{(-1)^{n_2}-(-1)^{n_1}}{8},
\end{align*}
and this is clearly positive if $n_2>n_1.$\end{proof}
\end{comment}

Consider first the monomial $+y^{a_0+1}x^{a_i}, i=1,2.$ 
Note the crucial fact that from the dependency relations, if either of these monomials contributes a nonzero coefficient, so does the preceding \textit{negative} monomial $-y^{a_0+1}x^{b-1}.$
\begin{lemma} Let $i=1,2.$ Then the net contribution of the monomials 
$y^{a_0+1}(-x^{b-1}+x^{a_i})$ to $g_{\mu,\nu,\lambda}$ is negative or zero. 
%modified 2018-11-23
\begin{proof}
The value  contributed to $g_{\mu,\nu,\lambda}$ 
by the monomial $+y^{a_0+1}x^{a_i}, i=1,2,$ is the coefficient 
$[x^{\mu_2} y^{\nu_2}]$ in the product $+y^{a_0+1}x^{a_i}\bar{F}_{2,2},$ which in turn is given by the vector partition function
\begin{equation}\label{pS1}p_S(\nu_2-(a_0+1), \nu_2-(a_0+1)+(\mu_2-a_i)).\end{equation}
On the other hand, the contribution from the negative monomial 
$-y^{a_0+1}x^{b-1}$ was shown in the proof of Theorem~\ref{7termF22Alternant} to be coming from Region II in Figure~\ref{fig:chambers}.  It therefore contributes the value
\begin{equation}\label{pSneg2}-p_S(\nu_2-(a_0+1), \nu_2-(a_0+1)+(\mu_2-b+1))=-{\nu_2-(a_0+1)\choose 2}.\end{equation}

But now Lemma~\ref{pSboundBinomCoeff} says the net contribution 
of these two monomials is negative or zero, 
%modified 2018-11-23
%\[\scriptstyle{-p_S(\nu_2-(a_0+1), \nu_2-(a_0+1)+(\mu_2-b+1))+p_S(\nu_2-(a_0+1), \nu_2-(a_0+1)+(\mu_2-a_i))<0},\]
as claimed.
\end{proof}
\end{lemma}

However, this is of course not sufficient to establish our theorem,  because \textit{both} positive monomials  $+y^{a_0+1}x^{a_i}, i=1,2$ can make a nonzero contribution. 
Appealing to the dependency relations, we see that a positive contribution from $+y^{a_0+1}x^{a_i}$ forces a 
 negative contribution from the monomial 
$-y^b x^{a_i},$ for each $i=1,2.$ 

\begin{lemma}\label{Case1A}  If $\mu_2-a_i\le 0,$ then the net contribution of $+y^{a_0+1}x^{a_i}$ and $-y^b x^{a_i}$ is negative or zero.  %modified 2018-11-23
\end{lemma}
\begin{proof} The contribution of $+y^{a_0+1}x^{a_i}$ is given by 
the vector partition function Eqn.~\eqref{pS1}, while that of $-y^b x^{a_i}$ is given by 
\begin{equation}\label{pSneg1}-p_S(\nu_2-b, \nu_2-b+(\mu_2-a_i)).\end{equation}
Because $\mu_2-a_i\le 0,$ in each case we have a vector partition function of the form $p_S(n,m)$ where $m<n$.  Hence each vector partition function corresponds to Region I in Figure~\ref{fig:chambers}.  But that function is clearly an increasing function of its second argument, $m.$  Also, we know from the inequalities~\eqref{expIneq} 
above that $\nu_2-(a_0+1)+(\mu_2-a_i)<\nu_2-b+(\mu_2-a_i).$  Hence the claim follows.
\end{proof}

\begin{lemma}\label{Case1Bi}  If $\mu_2-a_i> 0,$ then the net contribution of $+y^{a_0+1}x^{a_i}$ and $-y^b x^{a_i}$ is negative or zero.
\end{lemma}
\begin{proof} We must again carefully examine the respective contributions of these two monomials, which are 
\begin{equation}p_S(\nu_2-(a_0+1), \nu_2-(a_0+1)+(\mu_2-a_i)).\end{equation}%\label{pS1}
and 
\begin{equation}-p_S(\nu_2-b, \nu_2-b+(\mu_2-a_i).\end{equation}
%\label{pSneg1}
Each function above is of the form $p_S(n,m)$ where $n<m,$ so it is evaluated according to the formula for Region II or Region III in Figure~\ref{fig:chambers}. 
We know $\nu_2-(a_0+1)<\nu_2-b.$
We have three cases to consider:

\noindent
\textbf{Case 1:} Assume $\nu_2-(a_0+1)<\nu_2-b<\mu_2-a_i.$
Then each vector partition function above corresponds to Region II in 
Figure~\ref{fig:chambers}, given by a binomial coefficient so the net contribution is a difference of two binomial coefficients
${\nu_2-(a_0+1)+2\choose 2}-{\nu_2-b+2\choose 2},$
and this is clearly negative in view of the inequality~\eqref{expIneq}.

\noindent
\textbf{Case 2:} Assume $0\leq \nu_2-(a_0+1)< \mu_2-a_i
<\nu_2-b.$  Set $\mu_2-a_i=k, n_1=\nu_2-(a_0+1), n_2=\nu_2-b.$  
Thus we have $0\le n_1<k<n_2.$   In particular, $n_1-n_2\leq -2.$

Since $2n_1<n_1+k,$ we know that $p_S(n_1, n_1+k),$ which is the value of the contribution from the monomial $y^{a_0+1}x^{a_i}$, 
is specified by Region II in Figure~\ref{fig:chambers}, and is therefore given by the binomial coefficient ${n_1+2\choose 2}.$ 

Since $n_2\in (\frac{n_2+k}{2}, n_2+k),$  we conclude similarly that 
the contribution from the monomial $y^{b}x^{a_i}$ is given by computing $p_S(n_2, n_2+k)$ using the formula for Region III.

Hence, using the expression for $p_S(n,n+k)$ for Region III in Figure~\ref{fig:chambers}, the net contribution of $-y^b x^{a_i}+y^{a_0+1} x^{a_i}$ is given by 
\begin{align*} &p_S(n_1, n_1+k)-p_S(n_2, n_2+k)\\
&=\frac{n_1^2+3n_1+2}{2} -\left(\frac{n_2^2}{4}+n_2(\frac{k}{2}+1) -\frac{k^2}{4}+\frac{k}{2}+c(n_2+k) \right)
\end{align*}
Consider the function 
$f(k)=\frac{n_1^2+3n_1+2}{2} -\left(\frac{n_2^2}{4}+n_2(\frac{k}{2}+1) -\frac{k^2}{4}+\frac{k}{2} \right),$ 
 a polynomial in $k.$ 
 %defined in the interval $[n_1, n_2-1].$  (Although $n_1<k,$ as a polynomial in $k,$ $f(k)$ is defined in the interval $[n_1, n_2-1].$)
It is easy to check that $f'(k)=\frac{1}{2}(k-1-n_2)\leq -1 \text{ when } k<n_2,$ 
and hence this is a decreasing function of $k$ with maximum value $f(n_1)$ in the interval $[n_1, n_2-1].$ But 
\begin{align*} &f(n_1)=\frac{n_1^2}{2}-\frac{n_2^2}{4}+\frac{3n_1}{2}-\frac{n_1n_2}{2}-n_2+1+\frac{n_1^2}{4}-\frac{n_1}{2}\\
&=\frac{n_1^2-n_2^2}{4}+(n_1-n_2) +\frac{n_1^2}{2}-\frac{n_1n_2}{2}+1\\
&=(n_1-n_2)[\frac{n_1+n_2}{4}+1+\frac{n_1}{2}]+1
\end{align*}
Since $n_1-n_2\leq -2,$ and the expression in square brackets is at least $\frac{5}{4},$ we see that $f(k)\leq f(n_1+1)<f(n_1)<-\frac{3}{2}.$ 
To find the net contribution of the two monomials, we need to add the value of $c(n_2+k).$ But this is at most 1. It follows that the net contribution is negative.

\noindent
\textbf{Case 3:}   Assume $0 < \mu_2-a_i\leq \nu_2-(a_0+1)
<\nu_2-b.$
Again set $\mu_2-a_i=k, n_1=\nu_2-(a_0+1), n_2=\nu_2-b.$  
The contribution of the monomial $+y^{a_0+1}x^{a_i}$ is $p_S(n_1, n_1+k)$ while that of the monomial $-y^bx^{a_i}$ is 
$-p_S(n_2, n_2+k).$  The inequalities imply that the function $p_S$ corresponds to Region III in Figure~\ref{fig:chambers} in both cases.  Hence Lemma~\ref{pSRegionIII} applies (because $0< k\leq n_1<n_2$), showing that 
the net contribution, $p_S(n_1, n_1+k)-p_S(n_2, n_2+k),$ is indeed negative or zero. %modified 2018-11-23 and again 2018-11-25
\end{proof}

It remains to consider what happens when the last monomial with positive coefficient in the first line of $P_\lambda,$ $y^b x^{a_3},$ contributes to the Kronecker coefficient.  From the dependency relations, we know that then all the monomials  $y^b x^{a_i}$ must contribute nonzero terms as well, and possibly also one or both monomials 
$y^{a_0+1}x^{a_i}$.   In the latter  case there is also necessarily a negative contribution from $-y^{a_0+1}x^{b}.$

\begin{lemma}\label{lastmonomial}
For each of $i=1,2,$ the net contribution of the two monomials $-y^bx^{a_i}+y^bx^{a_3}$ is always negative or zero.  %modified 2018-11-23
\end{lemma}

\begin{proof}  Set $n=\nu_2-b, m_i=n+(\mu_2-a_i), i=1,2,3.$ The contribution of $y^bx^{a_3}$ is $p_S(n, m_3)$, and that of 
$-y^bx^{a_i}, i=1,2,$ is $-p_S(n, m_i).$  Note that $m_3<m_i, i=1,2,$ in view of ~\eqref{expIneq}.

We will examine the behaviour of the function $p_S(n,m)$ according to where $n$ falls in each of the intervals below.  Although there are two categories:

\begin{center} $0<\frac{m_3}{2}<m_3\leq \frac{m_i}{2}<m_i,$
or
$0<\frac{m_3}{2}< \frac{m_i}{2}<m_3<m_i,$\end{center}
  both can be treated by the same arguments, because the same difference of vector partition functions $p_S(n,m)$ comes into play in each case.

\noindent
\textbf{Case 1:} 
If $n\leq \frac{m_3}{2},$ then in either category, both $p_S(n, m_3)$ and $p_S(n, m_i)$ are computed by the formula for Region II in Figure~\ref{fig:chambers}, and hence both equal the binomial coefficient ${n+2\choose 2}.$ The net contribution of $-y^bx^{a_i}+y^bx^{a_3}$ here is zero.

\noindent
\textbf{Case 2:} If $n>m_i,$ then in either category, both $p_S(n, m_3)$ and $p_S(n, m_i)$ are computed by the formula for Region I in Figure~\ref{fig:chambers}. But the quasipolynomial for Region I is clearly an increasing function of the second argument of $p_S,$ and hence, (since $m_3<m_i$), $-p_S(n, m_i)+p_S(n, m_3)$ is negative or zero.  %modified 2018-11-23

\noindent
\textbf{Case 3:} Suppose $\frac{m_i}{2}\leq n\leq m_3. $ Then both functions $p_S$ correspond to Region III, 
and Lemma~\ref{pSRegionIII2nd} applies directly to show that $-p_S(n, m_i)+p_S(n, m_3)$ is negative or zero. %modified 2018-11-23

\noindent
\textbf{Case 4:} Suppose $m_3\leq n\leq  \frac{m_i}{2}. $ Then 
the monomial $-y^bx^{a_i}$ contributes $p_S(n, m_i)$ which is now a binomial coefficient since $n\leq \frac{m_i}{2}.$  By Lemma~\ref{pSboundBinomCoeff}, the net contribution here is negative or zero.  %modified 2018-11-23

\noindent
\textbf{Case 5:}  Suppose 
$0<\frac{m_3}{2}<m_3\leq \frac{m_i}{2}<n<m_i.$

We need to examine the difference 
$p_S(n,m_3)-p_S(n,m_i),$ where the first function corresponds to Region I and the second to region III.
We will consider the function 
$f(n)=p_S(n,m_3)-p_S(n,m_i)$ on the interval $[m_i/2, m_i].$ 
We have 
\[f(n)=\frac{m_3^2}{4}+m_3+c(m_3) 
-(nm_i-\frac{n^2}{2}-\frac{m_i^2}{4}+\frac{n+m_i}{2}+c(m_i)).\]
One checks that $f'(n)=-(m_i-n+1/2)\le -\frac{3}{2},$ and hence the function is decreasing with maximum at $\frac{m_i}{2}.$
This value is checked to be 
\[f(\frac{m_i}{2})= \frac{1}{4}(m_3^2-\frac{m_i^2}{2}) + (m_3-\frac{3m_i}{4}) +c(m_3)-c(m_i).\]
But $c(m_3)-c(m_i)\leq \frac{1}{4}$ as before, and we have $m_3<m_i.$  Hence $f(\frac{m_i}{2})<-\frac{1}{2}$ is negative, 
and so is $f(n).$

\noindent
\textbf{Case 6:}  Suppose 
$0<\frac{m_3}{2}< \frac{m_i}{2}<m_3<n<m_i.$ 
Exactly the same argument applies to this case, since we still have 
$m_3<n<m_i,$ which was the only inequality we used in the preceding argument.   This completes the proof of the lemma.
\end{proof}

\section*{Acknowledgements}
 The authors are grateful to Mich{\`e}le Vergne for in-depth discussions  and comparisons to previous work; in particular, for pointing out to us that remarkably, the vector partition function $F_{n,m}$ that we obtain from our elementary approach is in fact implicit in the work of Baldoni, Vergne and Walter \cite{BaldoniVergneWalter} and  \cite{SzenesVergne}.  
 We are grateful to Francesco Iachello for pointing out the relevance of the restriction of $GL(4)$ to $GL(2)\times GL(2)$ studied in this paper to nuclear physics.

This research was initiated at the Banff International Research Station in the Women in Algebraic Combinatorics II workshop.  It was further supported by the  National Science Foundation under Grant No. DMS-1440140 while the  first and third authors were in residence at the Mathematical Sciences Research
Institute in Berkeley, California, during the summer of 2017. The authors express their gratitude for the support of these two sources, without which this work would not have been possible. The research of MM is partially supported by National Science and Engineering Research Council Discovery Grant RGPIN-04157. During the completion of this work, she was hosted by Institut Denis Poisson (Tours, France) and LaBRI (Bordeaux, France).  The research of MR is partially supported by MTM2016-75024-P, FEDER, and the Junta de Andalucia under grants P12-FQM-2696 and FQM-333.
  
\def\cprime{$'$}

{\small\noindent{Marni Mishna, Dept. Mathematics, Simon Fraser University, Burnaby Canada {\tt mmishna$@$sfu.ca}}\\
\noindent{Mercedes Rosas, Dept. Algebra, Universidad de Sevilla, Sevilla  Espa\~na {\tt mrosas$@$us.es}}\\
\noindent{Sheila Sundaram, Pierrepont School, Westport, CT, USA {\tt shsund$@$comcast.net}}
}
\end{document}